\documentclass{amsart}
\usepackage{xy, enumerate, amssymb, amsfonts, amsbsy, amsthm, amsmath, amscd, epsfig, latexsym, mathtools, stmaryrd, epic, eepic, eucal, multicol, fancyhdr, graphicx}

\newcommand{\rrvert}{\vert}
\newcommand{\N}{\mathbb N}
\newcommand{\Ocal}{\mathcal O}
\newcommand{\Pro}{\mathbb P}
\newcommand{\Pcal}{\mathcal P}
\newcommand{\Q}{\mathbb Q}
\newcommand{\Qcal}{\mathcal Q}
\newcommand{\Z}{\mathbb Z}
\newcommand{\GL}{\mathrm{GL}}
\newcommand{\Sym}{\mathrm{Sym}}
\newcommand{\Spec}{\mathrm{Spec}}
\newcommand{\pr}{\mathrm{pr}}
\newcommand{\ov}{\overline}
\newcommand{\im}{\operatorname{Im}}
\newcommand{\wh}{\widehat}
\newcommand{\wt}{\widetilde}
\newcommand{\ev}{\mathrm{ev}}
\newcommand{\Char}{\mathrm{char}}
\theoremstyle{plain}
\newtheorem{theorem}{Theorem}[section]
\newtheorem*{maintheorem}{Main Theorem}
\newtheorem{lemma}[theorem]{Lemma}
\newtheorem{propos}[theorem]{Proposition}
\newtheorem{corol}[theorem]{Corollary}

\theoremstyle{remark}
\newtheorem{defi}[theorem]{Definition}
\newtheorem{rmk}[theorem]{Remark}

\input xy
\xyoption{all}

\title{The Chow Ring of the Stack of Smooth Plane Cubics.}
\author{Damiano Fulghesu}
 \thanks{The first author was partially supported by Scuola Normale Superiore and by Simons  Foundation grant \#360311. The second author was partially supported by research funds from the Scuola Normale Superiore (project SNS13VISTB) and from the PRIN project ``Geometria delle Variet\`a Algebriche''}
 \author{Angelo Vistoli}
\address[Damiano Fulghesu]{Department of Mathematics, Minnesota State University, 1104 7th Ave South, Moorhead, MN 56563, U.S.A.}\email{fulghesu@mnstate.edu}
\address[Angelo Vistoli]{Scuola Normale Superiore, Piazza dei Cavalieri 7, 56126 Pisa, Italy}\email{angelo.vistoli@sns.it}


\begin{document}

\begin{abstract}
In this paper we give an explicit presentation of the integral Chow ring of a stack of smooth plane cubics. We also determine some relations in the general case of hypersurfaces of any dimension and degree.
\end{abstract}

\maketitle

\section{Introduction}

Equivariant intersection theory was introduced by Edidin and Graham \cite
{EG}; it is of considerable interest, as it gives an intrinsic
integer-valued intersection theory on quotient stacks. In particular,
if $\mathcal{X}$ is a quotient stack $[U/G]$, where $U$ is a smooth
scheme of finite type over a field~$k$ and $G$ is an affine algebraic
group on $k$, then we obtain a Chow ring $A^{*}_{G}(U) = A^{*}(\mathcal
{X})$, which only depends on $\mathcal{X}$ and not on the presentation
of $\mathcal{X}$ as a quotient stack. If $\mathcal{X}$ is
Deligne--Mumford, or, equivalently, the action of $G$ on $U$ has finite
reduced stabilizers, then $A^{*}(\mathcal{X})\otimes\mathbb{Q}$
coincides with the rational Chow ring of $\mathcal{X}$, which had been
earlier studied by several authors \cite{Mum,Gil,Vis89}.

The ring $A^{*}(\mathcal{X})$ is usually much harder to compute than
$A^{*}(\mathcal{X})\otimes\mathbb{Q}$; for example, consider the
moduli stack $\mathcal{M}_{g}$ of smooth curves of genus~$g \geq2$;
the ring $A^{*}(\mathcal{M}_{g})$ has been computed only for $g = 2$
\cite{Vis98} (notice that, in this case, $A^{*}(\mathcal
{M}_{2})\otimes\mathbb{Q} = \mathbb{Q}$), whereas $A^{*}(\mathcal
{M}_{g})\otimes\mathbb{Q}$ is known for $g \leq6$ \cite{Mum,Fab,Iza,PV}, and, more importantly, it is the subject of an extensive
theory that has no parallel for integral Chow rings.

On the positive side, the ring $A^{*}(\mathcal{X})$ has been computed
when $\mathcal{X}$ is the stack of smooth hyperelliptic curves of
genus~$g$ when $g$ is a positive even number, and when $\mathcal{X}$
is the stack of rational nodal curves with at most one node.

In all these calculations the essential point is the determination of
the Chow ring of certain stacks of hypersurfaces. More precisely, let
$n$ and $d$ be positive integers. We define a stack $\mathcal
{X}_{n,d}$ as follows: an object of $\mathcal{X}_{n,d}$ over a
$k$-scheme $S$ consists of a vector bundle $F$ of rank~$n$, and a
Cartier divisor $X \subseteq\mathbb{P}(F)$ whose restriction to every
fiber is a smooth hypersurface of degree~$d$ (here, as everywhere else,
we follow \cite{Ful} and use the classic convention for the
projectivization of a vector bundle, so our $\mathbb{P}(F)$ would be
denoted by $\mathbb{P}(F^{\vee})$ in Grothendieck's convention).

An alternate description of $\mathcal{X}_{n,d}$ is as follows. Denote
by $W_{n,d}$ the vector space of homogeneous polynomials of degree~$d$
in $n$ variables with its natural action of $\GL_{n}$. Set $P_{n,d} =
\mathbb{P}(W_{n,d})$; so $P_{n,d}$ is the projective space of
hypersurfaces of degree $d$ in $\mathbb{P}^{n-1}$. If $Z \subseteq
P_{n,d}$ is the discriminant locus, then we have
\[
\mathcal{X}_{n,d} = [(P_{n,d} \setminus Z)/
\GL_{n}] .
\]
By standard facts of equivariant intersection theory this gives a set
of generators for the ring $A^{*}(\mathcal{X}_{n,d}) = A^{*}_{\GL
_{n}}(P_{n,d} \setminus Z)$, which are the Chern classes $c_{1},
\dots, c_{n}$ of the tautological representation of $\GL_{n}$, and
$h = c_{1} (\mathcal{O}_{P_{n,d}}(1) )$. The relations among
these generators $c_{1}, \dots, c_{n}$ and $h$ are obtained from the
classes of the image of the pushforward $A^{\GL_{n}}_*(Z) \to A^*_{\GL
_{n}}(P_{n,d})$.

A set of natural relations is obtained as follows. Let $\widetilde{Z}
\subseteq P_{n,d} \times\mathbb{P}^{n-1}$ be the reduced subscheme
consisting of pairs $(X, p)$, where $X$ is a hypersurface of degree~$d$
in $\mathbb{P}^{n-1}$, and $p$ is a singular point of $X$. Then $Z$ is
the image of $\widetilde{Z}$ in $P_{n,d}$, and hence every class in
$A^{*}_{\GL_{n}}(\widetilde{Z})$ when pushed down to $A^{*}_{\GL
_{n}}(P_{n,d})$ gives a relation in $A^{*}_{\GL_{n}}(P_{n,d})$. The
image of the pushforward $A^{*}_{\GL_{n}}(\widetilde{Z}) \to
A^{*}_{\GL_{n}}(P_{n,d})$ is easily determined (see Theorem~\ref
{generators.I.Z.tilde}); this gives certain relations $\alpha_{1}$,
\dots,~$\alpha_{n} \in\mathbb{Z}[c_{1}, \dots, c_{n}, h]$. When $d
= 2$ or $n = 2$, it is proved in \cite{Pan,EF,EFh} that $\alpha
_{1}$, \dots,~$\alpha_{n}$ generate the ideal of relations, so that
(Theorem \ref{lower.cases})
\[
A^{*}(\mathcal{X}_{n,d}) = \mathbb{Z}[c_{1},
\dots, c_{n}, h]/(\alpha_{1}, \dots, \alpha_{n})
 .
\]
With rational coefficients, it is easy to verify that the classes
$\alpha_{i}$ generate the ideal of relations of the generators in
$A^{*}(\mathcal{X}_{n,d})\otimes\mathbb{Q}$ (see Remark \ref
{rational.coefficients}).

The main purpose of this paper is to investigate the first case that is
not covered by Theorem \ref{lower.cases}, namely $A^{*}(\mathcal
{X}_{3,3})$. The stack $\mathcal{X}_{3,3}$ can alternatively be
thought of as the stack in which an object $(C, L)$ over a $k$-scheme
$S$ is a family $C \to S$ of smooth curves of genus~$1$ together with
an invertible sheaf $L$ on $X$ whose degree when restricted to every
fiber is 3.

It turns out the $\alpha_{i}$'s are not sufficient to generate the
whole ideal of relations, but we need to add a polynomial $\delta_{2}
\in\mathbb{Z}[c_{1}, c_{2}, c_{3}, h]$ of degree~$2$ with the
property that $2\delta_{2} \in(\alpha_{1}, \alpha_{2}, \alpha
_{3})$. The following is our main result.

\begin{maintheorem}
The ring $A^{*}(\mathcal{X}_{3,3})$ is the quotient
\[
\mathbb{Z}[c_{1}, c_{2}, c_{3}, h]/(
\alpha_{1}, \alpha_{2}, \alpha _{3},
\delta_{2}),
\]
where
\begin{align*}
\begin{split} \alpha_1&=12(h - c_1),
\\
\alpha_2&= 6h^2 - 4h c_1 - 6
c_2,
\\
\alpha_3&= h^3 - h^2c_1 + h
c_2-9c_3,
\\
\delta_2 &=21h^2 - 42h c_1 +
9c_2 + 18c_1^2. \end{split} %
\end{align*}

We have $\delta_{2} \notin(\alpha_{1}, \alpha_{2}, \alpha_{3})$,
whereas $2\delta_{2} \in(\alpha_{1}, \alpha_{2}, \alpha_{3})$.
\end{maintheorem}

The next natural case to be considered is $A^*(\mathcal{X}_{3,4})$,
which is a current work in progress. This is particularly interesting,
since it would allow us to determine the integral Chow ring of
$\mathcal M_3 \setminus\mathcal H_3$, that is, the stack of
smooth nonhyperelliptic curves of genus 3.

\subsection*{Strategy of Proof and Description of Content}
Section~\ref{preliminaries} introduces the basic notation and reviews
some general results about $\GL_n$-equivariant Chow groups, which
constitute the fundamental tools used in the proof of the Main Theorem.
In particular, we give a reduction result to torus action (Lemma \ref
{algebraic.lemma}), an explicit form of torus equivariant hyperplane
classes (Lemma \ref{hyperplane.T.classes}), and an explicit
localization theorem for torus actions (Theorem \ref{explicit.localization}).

The real action starts in Section~\ref{universal-locus}; here, and in
the following section, we give some general results on $A^{*}(\mathcal
X_{n,d})$. We set up the notation and give a formula for the class in
$A^{*}_{\GL_{n}}(P_{n,d})$ of hypersurfaces that split as sums of $s$
hypersurfaces of degrees~$d_{1}$, \dots,~$d_{s}$ with $d_{1} + \cdots+
d_{s} = d$ (Theorem~\ref{delta.classes}).

In the next section, we study the ideal $I_{\widetilde{Z}} \subseteq
A^{*}(P_{n,d})$, which is the image of the pushforward $A^{*}_{\GL
_{n}}(\widetilde{Z})$. We give explicit formulas for the generators
$\alpha_{1}$, \dots,~$\alpha_{n}$ (Theorem~\ref
{generators.I.Z.tilde}), show that
\[
A^{*}_{\GL_{n}}(P_{n,d}) \otimes\mathbb Q = \mathbb
Q[c_{1}, \dots, c_{n}, h]/(\alpha_{1}, \dots,
\alpha_{n})
\]
(Remark \ref{rational.coefficients}), and prove that the intrinsic
relation satisfied by the hyperplane class $h$ is in $I_{\widetilde
{Z}}$ (Proposition~\ref{big.polynomial}).

Next, we specialize to the case $n = d = 3$ in the last section,
Section~\ref{plane.cubics}. We write down the classes $\alpha_{1}$,
$\alpha_{2}$, and $\alpha_{3}$ as they come out of Theorem~\ref
{generators.I.Z.tilde}.

We use Theorem~\ref{delta.classes} to show that $\delta_{2}$
represents the class of the locus $Z_{2} \subseteq P_{3,3}$ of
reducible cubics; then we need to show that the image of the
pushforward $i_{*}: A_*^{\GL_{n}}(Z) \to A^{*}_{\GL
_{n}}(P_{n,d})$ is in the ideal generated by $(\alpha_{1}, \alpha
_{2}, \alpha_{3}, \delta_{2})$.

For this purpose, we define a stratification of the discriminant locus
$Z$ (Definition~\ref{strata}). If $T$ is one of the strata, denote by
$\overline{T}$ its closure in $P_{3,3}$. For each $T$, we need to show
that every class in $A^{*}_{\GL_{3}}(P_{3,3})$ supported in $\overline
{T}$ is the sum of a class in $(\alpha_{1}, \alpha_{2}, \alpha_{3},
\delta_{2})$ and a class supported in $\overline{T} \setminus T$;
then the result follows by descending induction on the codimension of
the strata (see Section~\ref{basic.principle} for a fuller
explanation). In some cases the elements of $T$ have a distinguished
singular point (for example, this happens for the stratum $Z_{(3,1)}$
consisting of unions of a smooth conic and a line that is tangent to a
point); this gives a lifting of the embedding $T \subseteq P_{3,3}$ to
a morphism $T \to\widetilde{Z}$, which makes what we need to prove
obvious. In other cases, we produce a finite $\GL_{3}$-equivariant
morphism $\overline{T}_{1} \to\overline{T}$. We have to use ad hoc
arguments, together with Theorem~\ref{delta.classes}, to show that the
image of the pushforward $A_*^{\GL_{3}}(\overline{T}) \to A^*_{\GL
_{3}}(P_{3,3})$ coincides with the image of $A^*_{\GL_{3}}(\overline
{T}_{1}) \to A^*_{\GL_{3}}(P_{3,3})$ up to classes supported in
$\overline{T} \setminus T$. Finally, we show that the image of
$A^*_{\GL_{3}}(\overline{T}_{1}) \to A^*_{\GL_{3}}(P_{3,3})$ is
contained in $(\alpha_{1}, \alpha_{2}, \alpha_{3}, \delta_{2})$.

More precisely, Section~\ref{Z31} is dedicated to the proof of the
fact that classes in $A^{*}_{\GL_{3}}(P_{3,3})$ supported in the
closure of the locus of conics together with a tangent line are in
$I_{\widetilde{Z}}$. Section~\ref{Z32} contains the hardest step
of the proof, the fact that classes supported in the locus consisting
of sums of three lines are in $I_{\widetilde{Z}}$; this is technically
rather involved. In Section~\ref{Z2}, we compute the class $\delta
_{2}$ of the locus of reducible curves, and we show that the classes
supported in this locus are in $(\alpha_{1}, \alpha_{2}, \alpha_{3},
\delta_{2})$.

The proof of the Main Theorem is finally concluded in Section~\ref{conclusion}.

Some of the calculations have been carried out by using Maple 16.

\section{Preliminaries on $\GL_n$-Equivariant Chow Groups}\label
{preliminaries}

\subsection{Intersection Ring of $B \GL_n$}
We work on a base field $k$ of characteristic 0 or greater than a fixed
integer $d \geq2$.
Let $n \geq2$ be another integer, and let $E$ be the standard
representation of $\GL_n$. The stack $[E^{\vee}/ \GL_n]$ is a vector
bundle over $B \GL_n$ whose Chern roots are $l_1, \dots, l_n$. On the
other hand, let $c_1, \dots, c_n$ be the Chern classes of $[E/\GL
_n]$, so we have
\begin{eqnarray*}
c_1 &=&-(l_1 + \cdots+ l_n),
\\
&\dots &
\\
c_i &=& (-1)^i s_i(l_1, \dots,
l_n),
\\
&\dots &
\\
c_n &=& (-1)^n l_1\cdots l_n,
\end{eqnarray*}
where $s_i ( x_1, \dots, x_n  )$ is the $i$th symmetric
polynomial in $n$ variables.
Let $T$ be the maximal torus for $\GL_n$ represented by diagonal
matrices. The total character of the $T$-module $E$ can be expressed as
a sum of linearly independent characters $\lambda_1, \dots, \lambda
_n$, and therefore we have $A^{*}_T=\Z [ c_1 ( \lambda_1
 ), \dots, c_1 ( \lambda_n  )  ]$. According
to our notation, we have $l_i=-c_1 ( \lambda_i  )$, and we
will identify $A^*_T$ with $\Z [ l_1, \dots, l_n  ]$.
Similarly, we can see the Weyl group $S_n$ as acting on $A^*_T$ by
permuting the classes $l_i$. and we have $A^*_{\GL_n}= ( A^*_T
 )^{S_n}$.

\subsection{Reduction to the Torus Action} Let $X$ be a smooth
$G$-space, that is, $X$ is a smooth algebraic space with an action of $G$,
\[
\alpha: G \times X \to X,
\]
where $\alpha$ is also a morphism of algebraic spaces.

Throughout this paper, we consider equivariant intersection Chow rings
of the form $A^*_G(X)$ where $G$ will be the group $T$ or ${\GL_n}$.
In the cases we are interested, the ring $A^*_G(X)$ will have the
structure of a finitely generated $A^*_G$-algebra. In particular, there
is an isomorphism
\[
A^*_{G}(X) \cong\frac{A^*_G[x_1, x_2, \dots, x_r]}{I}
\]
for a suitable set of variables $x_1, \dots, x_r$ and a suitable ideal
$I$ called {\it ideal of relations}. For this reason, we will usually
write a class in $A^*_G(X)$ as a polynomial in several variables with
coefficients in $A^*_G$ leaving the ideal of relations implicit. We
will use the following algebraic lemma.

\begin{lemma}\label{algebraic.lemma}
Let $G$ be a special algebraic group, and let $T \subset G$ be a
maximal torus.
\begin{enumerate}[II)]
\item[I)] Let $X$ be a smooth $G$-space. Let $I \subset A^*_{G} ( X
 )$ be an ideal. Then
\[
I A^*_{T} ( X ) \cap A^*_{G} ( X ) = I.
\]
\item[II)] Let $\{ x_1, \dots, x_r \}$ be a set of variables, and let $I
\subset A^*_{G} [ x_1, \dots, x_r  ] $ be an ideal. Then
\[
I A^*_{T} [ x_1, \dots, x_r ] \cap
A^*_{G} [ x_1, \dots, x_r ] = I.
\]
\end{enumerate}
\end{lemma}

\begin{proof}
From \cite[Proposition~2.2]{EFh} we have that $A^*_{G} ( X  )$ is
a (noncanonical) summand of the $A^*_{G} ( X  )$-module
$A^*_{T} ( X  )$. Now, in general, if $R \subset S$ is a
ring extension, $R$ is a summand of $S$, and $I \subset R$ is an ideal,
then $IS \cap R = I$ (see, for example, \cite[Propositions~9 and 10]{HE}).
This concludes part I).

For part II), we first apply part I) by considering $X=\Spec(k)$.
Consequently $A^*_{G}$ is a summand of $A^*_{T}$. Now let us write
$A^*_{T} \cong A^*_{G} \oplus M$ for some submodule $M$. Then we have
(to prove this, use induction on $r$ in \cite[Chapter~2, Example~6]{AM})
\[
A^*_{T} [ x_1, \dots, x_r ] \cong
A^*_{G} [ x_1, \dots, x_r ]
\otimes_{A^*_G} A^*_T \cong A^*_{G} [
x_1, \dots, x_r ] \oplus M [ x_1, \dots,
x_r ].
\]
\end{proof}

\begin{rmk}\label{remark.algebraic.lemma}
We will apply the above Lemma \ref{algebraic.lemma} in the following
way. Let $\gamma\in A^*_{\GL_n} ( X  )$ (resp. $\gamma\in
A^*_{\GL_n} [ x_1, \dots, x_r  ]$), and let $I \subset
A^*_{\GL_n} ( X  )$ (resp. $I \subset A^*_{\GL_n} [
x_1, \dots, x_r  ]$) be an ideal. If $\gamma\in IA^*_{T} (
X  )$, then $\gamma\in I$.
\end{rmk}

\subsection{Equivariant Intersection Theory on Projective Spaces}

Let $W$ be a $\GL_n$-representation of dimension $q$. The vector space
$W$ is equipped with an induced $T$-action. We have also a canonical
action of $\GL_n$ (resp.\ $T$) on $\Pro ( W )$. Let $G$ be
either $\GL_n$ or $T$. Let $h=c^G_1 ( \Ocal_{\Pro(W)}(1)
)$, and let $r_1, \dots, r_k$ be the Chern roots of $W$. We have an
exact sequence of $A^*_G$-modules (see \cite[Lemma 2.3]{EFh})
\begin{equation}
\label{exact.sequence} \xymatrix{
0 \ar[r] & J \ar[r] & A^*_{G}[x] \ar[rr]_(0.45){\ev_h} &
&A^*_{G} ( \Pro(W)  ) \ar[r] & 0,
}
\end{equation}
where $\ev_h$ is the evaluation morphism at $h$, and the ideal of
relations $J$ is the principal ideal generated by the polynomial
$P(x):=\prod_{i=1}^q  ( x + r_i  )$. Notice that the exact
sequence (\ref{exact.sequence}) induces the isomorphism
\[
A^*_{G} ( \Pro(W) ) \cong\frac{A^*_{G}[x]}{J}.
\]

\begin{rmk}\label{linear.combinations}
The Chern roots $r_i$ are linear combinations of $l_1, \dots, l_n$
with integral coefficients. However, $P(x)$ can be written as a
polynomial with coefficients in $A^*_{\GL_n}$.
\end{rmk}

Since $P(x)$ is a monic polynomial of degree $q$, $A^*_{G}  ( \Pro(W)
)$ is an $A^*_{G}$-module freely generated by the set $\{ h^i|   0\leq
i < q \}$. So we can define a splitting morphism $\psi:
A^*_{G} ( \Pro(W)  ) \to A^*_{G}[x]$ of $A^*_G$-modules as follows.
\begin{defi}\label{splitting.morphism}
Let $\gamma\in A^*_{G}  ( \Pro(W)  )$. We define $\psi
(\gamma):=Q(x)$, where $Q(x)$ is the unique polynomial in $A^*_{G}[x]$
whose degree in $x$ is less than $q$ and $Q(h)=\gamma$.
\end{defi}

The polynomial $P(x)$ depends on the Chern roots $r_i$ of the
representation $W$. However, as mentioned in Remark \ref
{linear.combinations}, the Chern roots $r_i$ are linear combinations of
the classes $l_1, \dots, l_n$ with integral coefficients. The
polynomial $P(x)$ is uniquely determined by such integral coefficients,
and it has a combinatorial flavor. The following notation is introduced
to write the polynomial $P(x)$ in an explicit form, which is
computationally useful.

We define the set $\Pcal_d$ of (unordered) partitions of $d$. For example,
\[
\Pcal_4:=\{ \{4\}, \{3,1\}, \{2, 2\}, \{2, 1, 1\}, \{1,1,1,1\} \}.
\]

\begin{defi}
Let $\mu\in\Pcal_d$. We define the set
\[
\N^n(\mu):= \{ v \in\N^n \mid   \mu=\{ v_i
\neq0\} \};
\]
in other words, we say that a vector $v \in\N^n$ is in $\N^n(\mu)$
if the set of nonzero entries of $v$ is $\mu$.
\end{defi}

For example, let $\mu:= \{ 3,1 \} \in\Pcal_4$. Then, by definition,
we have
\[
\N^3(\mu)= \{ (0,1, 3), (0, 3, 1), (1, 0, 3), (3, 0, 1), (1,3,0),
(3,1, 0) \}.
\]

\begin{defi}
For every natural number $q \in\N$, we define the set
\[
\N^n(q):= \{ v \in\N^n |   |v| = q \}.
\]
\end{defi}

For example, by definition, we have
\[
\N^3(2):= \{ (2,0,0), (0,2,0), (0,0,2), (1,1,0), (1,0,1), (0,1,1)
\}.
\]

\begin{defi}
Let $\mu\in\Pcal_d$. We define the following polynomial in $A^*_T[x]$:
\[
P_{\mu}(x):=\prod_{v   \in  \N^n(\mu)} (x + v \cdot l),
\]
where $l$ is the vector $\langle l_1, \dots, l_n \rangle$ in the free $\Z$-module
$ ( A^1_{T}  )^n$.
\end{defi}

For example, for $n=3$,
\begin{eqnarray*}
P_{\{ 2,1,1\} }(x)&=&(x + 2l_1 + l_2 + l_3)
(x + l_1 + 2l_2 + l_3) (x + l_1 +
l_2 + 2l_3);
\\
P_{\{ 3,1\}}(x) &=& (x+3l_1+l_2)
(x+3l_1+l_3) (x+3l_2+l_1)\\
&&{}\times
(x+3l_2+l_3) (x+3l_3+l_1)
(x+3l_3+l_2).
\end{eqnarray*}
Notice that, for every $\mu$, the polynomial $P_{\mu}(x)$ is
symmetric with respect to the classes $l_i$, and therefore it can be
written as an element of $\Z[c_1, \dots, c_n][x]=A^*_{\GL_n}[x]$. We
will effectively use the polynomials $P_{\mu}(x)$ in Proposition \ref
{big.polynomial}.

We now recall two results that will be used extensively later in order
to perform computations.

The following lemma will allow us to write explicitly the equivariant
class of a $T$-invariant hypersurface.

\begin{lemma}[{\cite[Lemma 2.4]{EFh}}]\label{hyperplane.T.classes}
 Let $H \subset\Pro(W)$ be a $T$-invariant
hypersurface defined by a homogeneous equation $F=0$ of degree $d$ such
that $z \cdot F=\chi^{-1}(z)F$ for some character $\chi:T \to\mathbb
G_m$. Then we have the following identity in $A^*_T  ( \Pro(W)
 )$:
\[
[ H ]_T = c_1^T ( \Ocal_{\Pro(W)}(d) )
+ c_1(\chi).
\]
\end{lemma}

The following theorem is also known as an {\it explicit localization formula}.

\begin{theorem}[{\cite[Theorem~2]{EGL}}]\label{explicit.localization}
 Define the $A^*_T$-module
\[
\Qcal:= ( ( A^*_T )^+ )^{-1}A^*_T,
\]
where $ (  ( A^*_T )^+ )^{-1}$ is the
multiplicative system of the reciprocals of homogeneous elements of
$A^*_T$ of positive degree.

Let $X$ be a smooth $T$-variety and consider the locus $F \subset X$ of
fixed points for the action of $T$. Let $\bigcup_{j \in I} F_j =F$ be
the decomposition of $F$ into irreducible components. For every $\gamma
\in A^T_*(X) \otimes\Qcal$, we have the identity
\[
\gamma=\sum_{j \in I} i_{j*}
\frac{i_j^* (\gamma)}{c^T_{\mathrm{top}} ( N_{F_j}X )},
\]
where, for all $j \in I$, the map $i_j$ is the inclusion $F_j \to X$,
and $N_{F_j}X$ is the normal bundle of $F_j$ in $X$.
\end{theorem}

In other words, Theorem \ref{explicit.localization} gives an explicit
formula for decomposing every class $\gamma$ in $A_*^T(X)$ in terms of
the pushforwards of the restrictions of $\gamma$ to the subvarieties
$F_j$ up to dividing by invertible elements in $\Qcal\otimes A^*_T(F_j)$.

\section{The Space of Hypersurfaces}\label{universal-locus}

\subsection{Resolution of the Degeneracy Locus}

Let $W_d:=\Sym^d(E^{\vee})$, and let $\Delta_d$ be the degeneracy
locus of singular $d$-forms. Let $N:=\dim_k W_d = \binom{n+d-1}{d}$.
We point out that, to simplify the notation, we denote by $W_d$ the
space $W_{n,d}$ (see the Introduction) with the implicit assumption
that $E$ is the standard representation of $\GL_n$. We fix a set of coordinates
\[
\{ a_v \}_{v \in\N^n \text{ s.t. } |v| = d},
\]
where $a_v$ represents the coefficient of the monomial $X^v$, and
$X=[X_0, \dots, X_{n-1}]$ is a coordinate system for $\Pro(E)$.

Define $Z:=\Pro(\Delta_d) \subset\Pro(W_d)$ and consider the
universal hypersurface $U \subset\Pro(W_d)\times\Pro(E)$. We have
the following equivariant projections:
\[
\xymatrix{
U \subset\Pro(W_d)\times\Pro(E) \ar[d]_{\pi_1} \ar[r]^{\qquad
\quad\pr} &\Pro(E)\\
\Pro(W_d) &
}
\]
The hypersurface $U$ is given by the bihomogeneous equation of bidegree $(1,d)$
\begin{equation*}
F(X):= \sum_{v \in\N^n(d)} a_{v} X^{v},
\end{equation*}
where $X=[X_0, \dots, X_{n-1}]$ is a coordinate system for $\Pro(E)$.

In $\Pro(W_d)\times\Pro(E)$, we also define the subvariety
$\widetilde{Z}$ given by equations $\{ F_{X_i} (X) \}_{i=0,\dots,
n-1}$ where $F_{X_i} (X)$ is the partial derivative of the polynomial
$F(X)$ with respect to the variable $X_i$. Notice that the restriction
morphism $\widetilde{Z} \xrightarrow{\pi_1} Z$ is generically $1:1$,
since the generic singular hypersurface has exactly one nodal point. An
easy dimensional argument shows that $\widetilde{Z}$ is a complete
intersection subvariety.

Now, with abuse of notation, we call $i$ both inclusion maps
$\widetilde{Z} \to\Pro(W_d) \times\Pro(E)$ and
$Z \to\Pro(W_d)$. Moreover, we define $h_d:=\pi_1^*(h_d)$ and
$t:=\pr^*(t)$, where $t$ is the hyperplane class of $\Ocal_{\Pro(E)}(1)$.

We would like to write explicit generators for the ideal $I_{Z}:=i_*
 ( A^*_{\GL_n} (Z)  )$ in terms of $c_1, \dots, c_n$, and
$h_d$. As a preliminary step, we determine generators for the ideal
$I_{\widetilde{Z}}:=\pi_{1*} i_*  ( A^*_{\GL_n} (\widetilde
{Z})  )$ (Section~\ref{ideal.alpha}). More precisely, we have
the inclusion $I_{\widetilde{Z}} \subseteq I_{Z}$. In the case of
quadrics ($d=2$) or effective divisors of the projective line ($n=2$),
the equality $I_{\widetilde{Z}} = I_{Z}$ holds (see Theorem \ref
{lower.cases}). However, this is not true in general, as we show by
determining $I_Z$ in the case of plane cubics (Section~\ref{plane.cubics}).

\subsection{Equivariant Classes of the Loci of Reducible Hypersurfaces}

On $\Pro(W_d)$, we have the natural action of $\GL_n$
\[
(A \cdot[f]) (X) = [f](A^{-1}X).
\]
Let $h_d$ be the hyperplane class associated with $\Ocal_{\Pro
(W_d)}(1)$. We have the splitting exact sequence
\begin{equation}
\label{PW} \xymatrix{
0 \ar[r] &  ( P_{[d]}(x)  ) \ar[r] & A^*_{\GL_n}[x] \ar
[rr]_{\ev_{h_d}} & &A^*_{\GL_n} ( \Pro(W_d)  ) \ar[r]
\ar@/_1pc/[ll]_{\psi}& 0,
}
\end{equation}
where $\psi$ is as in Definition \ref{splitting.morphism}, and
$P_{[d]}(x)$ is the polynomial
\[
P_{[d]}(x):=\prod_{\mu  \in  \Pcal_d}
P_\mu(x)= \prod_{v \in\N
^n(d)} (x + v \cdot l).
\]

Our next goal is to determine an explicit formula (see Theorem~\ref
{delta.classes}) for the equivariant classes of the loci of different
types of reducible hypersurfaces. To this end, we need to introduce
some notation.

Let $\mu\in\Pcal_d$ be an (unordered) partition of $d$. We will
think of $\mu$ either as a multiset or as an $s$-tuple $(k_1, \dots,
k_s)$, where
\begin{itemize}
\item$s$ is the number of elements of the multiset $\mu$,
\item$k_1 \leq k_2 \leq\cdots\leq k_s$, and
\item$k_1 + \cdots+ k_s = d$.
\end{itemize}

For every natural number $q$, we define $\mu(q)$ to be the frequency
of $q$ in $\mu$.

\begin{defi}
For every $\mu\in\Pcal_d$, we denote by $\delta_{\mu}$ the
equivariant class of the locus of reducible hypersurfaces of degree $d$
that are unions over the integers $q=1, \dots, d$ of $\mu(q)$
hypersurfaces of degree $q$.

We also define the variety
\[
W_{\mu} := \prod_{j=1}^s
\Pro(W_{k_j}) = \prod_{q=1}^d
\Pro (W_q)^{\mu(q)}
\]
and the product map
\begin{eqnarray*}
\pi_{\mu}: W_{\mu} & \to& \Pro(W_d),
\\
( f_1, \dots, f_s ) & \mapsto& f_1
f_2 \dots f_s.
\end{eqnarray*}
\end{defi}

It is worth noticing that the degree $\deg(\pi_{\mu})$ of the
product map $\pi_{\mu}$ is $\prod_{q=1}^d \mu(q)!$.

\begin{rmk}\label{classes.fixed.points} For every positive integer
$d$, the irreducible components of the fixed locus for the action of
$T$ on $\Pro(W_d)$ are the points $\{ Q_{v} \}_{v \in\N^n(d)}$,
where, for every $v \in\N^{n}(d)$, the only coordinate of $Q_v$
different from~0 is $a_v$. Each point $Q_v$ is the complete
intersection of the coordinate hyperplanes $a_{w}=0$ with $w \neq v$.
By Lemma \ref{hyperplane.T.classes} we obtain
\begin{equation*}
[ Q_{v} ] =  \frac{P_{[d]}(x)}{x + v \cdot l} \bigg\rrvert _{x=h_{d}}=\prod
_{w \in\N^n(d) \text{ s.t. } w \neq v} (h_d + v \cdot l).
\end{equation*}
\end{rmk}

\begin{lemma}\label{top.chern.class}
Let $v_0 \in\N^n(d)$. We have the following identity:
\[
c^T_{\mathrm{top}} ( T_{Q_{v_0}} \Pro(W_d) ) =
\prod_{v
\in\N^n(d) \text{ s.t. } v \neq v_0} (v - v_0) \cdot l.
\]
\end{lemma}

\begin{proof}
Since the coordinate $a_{v_0}$ of $Q_{v_0}$ is different from zero, we
can reduce our computations to local coordinates
\[
\biggl\{ \ov{a}_v:= \frac{a_v}{a_{v_0}} \biggr\}_{v \in\N^n(d)
\text{ s.t. } v \neq v_0}.
\]
Such coordinates are the same as the coordinates of the tangent space
at $Q_{v_0}$. Therefore, the action of $T$ on $T_{Q_{v_0}} \Pro(W_d)$ is
\[
t \cdot ( \ov{a}_v )_{v \in\N^n(d) \text{ s.t. } v
\neq v_0} = ( \lambda^{v_0 - v}(t)
\ov{a}_v )_{v \in\N
^n(d) \text{ s.t. } v \neq v_0},
\]
where $\lambda=(\lambda_1, \dots, \lambda_n)$ is the vector of
standard characters for the action of $T$ on~$E$. Consequently,
according to our notation, we get
\[
c^T_{\mathrm{top}} ( T_{Q_{v_0}} \Pro(W_d) ) =
\prod_{v
\in\N^n(d) \text{ s.t. } v \neq v_0} (v - v_0) \cdot l.
\]
%
\end{proof}

We are now ready to prove an explicit formula for the classes $\delta
_{\mu}$. We would like to point out that the following result holds in
general for hypersurfaces of any dimension and degree.

\begin{theorem}\label{delta.classes}
We have the following identity:
\begin{equation}
\label{delta.classes.formula} \delta_{\mu} = \frac{1}{\deg(\pi_{\mu})} \sum
_{(v_1, \dots,
v_s) \in\N^n(k_1) \times\cdots\times\N^n(k_s)} \frac
{ \prod_{v \in\N^n(d) \text{ s.t. } v \neq v_1 + \cdots
+ v_s} (h_d + v \cdot l)}{ \prod_{j=1}^s  ( \prod_{v \in\N^n(k_j) \text{ s.t. } v \neq v_j} (v - v_j) \cdot l
) }.
\end{equation}
\end{theorem}

\begin{proof}
Consider the map $\pi_{\mu}$. Since $\pi_{\mu*}(1) = \deg(\pi
_{\mu}) \delta_{\mu}$ and the ring\break  $A^*_{\GL_n}  ( \Pro
(W_d) )$ is torsion free, we have
\begin{equation}
\label{deltamu} \delta_{\mu} = \frac{1}{\deg(\pi_{\mu})} \pi_{\mu*}(1).
\end{equation}
Now, to determine $\pi_{\mu*}(1)$, we apply Theorem \ref
{explicit.localization}. First of all, notice that the locus of fixed
points of $W_{\mu}$ is the disjoint union of the points
\[
\{ ( Q_{v_1}, \dots, Q_{v_s} ) \}_{(v_1,
\dots, v_s) \in\N^n(k_1) \times\cdots\times\N^n(k_s)}.
\]
Consequently, by applying Theorem \ref{explicit.localization} we get
\[
1 = \sum_{(v_1, \dots, v_s) \in\N^n(k_1) \times\cdots\times\N
^n(k_s)} \frac{ [  ( Q_{v_1}, \dots, Q_{v_s}  )
 ]}{ \prod_{j=1}^s c^T_{\mathrm{top}} (
T_{Q_{v_j}} \Pro(W_{k_j})  ) }.
\]
Now, by applying Lemma \ref{top.chern.class} we have
\[
1= \sum_{(v_1, \dots, v_s) \in\N^n(k_1) \times\cdots\times\N
^n(k_s)} \frac{ [  ( Q_{v_1}, \dots, Q_{v_s}  )
 ]}{ \prod_{j=1}^s  ( \prod_{v \in\N^n(k_j)
\text{ s.t. } v \neq v_j} (v - v_j) \cdot l  ) }.
\]
Next, we evaluate $\pi_{\mu*}$ on both sides, and observing that $\pi
_{\mu}  ( Q_{v_1}, \dots, Q_{v_s}  )=Q_{v_1 + \cdots+
v_s}$, we get
\[
\pi_{\mu*}(1) = \sum_{(v_1, \dots, v_s) \in\N^n(k_1) \times\cdots
\times\N^n(k_s)}
\frac{ [ Q_{v_1 + \cdots+ v_s}
]}{ \prod_{j=1}^s  ( \prod_{v \in\N^n(k_j) \text
{ s.t. } v \neq v_j} (v - v_j) \cdot l  ) }.
\]

Finally, applying Remark \ref{classes.fixed.points} combined with
equation \eqref{deltamu}, we get formula (\ref{delta.classes.formula}).
\end{proof}

\section{The Ideal $I_{\widetilde{Z}}$}\label{ideal.alpha}

The main goal of this section is to determine a set of generators for
the ideal $I_{\widetilde{Z}}$ (see Section~\ref{universal-locus} for
basic definitions).

\begin{propos}\label{prod.Wd.E}
We have an exact sequence of $A^*_{\GL_n}$-modules
\[
\xymatrix@R=10pt{
0 \ar[r] &  (P_{[d]}(x), P_{\{1\}}(- y)  ) \\\ar[r] &
A^*_{\GL_n}[x,y] \ar[rr]_{\hspace*{-15pt}\ev_{(h_d,t)}} &&  A^*_{\GL_n}  (
\Pro(W_d) \times\Pro(E)  )  \ar@/_1pc/[ll]_{\psi}  \ar[r] & 0,
}
\]
where $\ev_{(h_d, t)}$ is the evaluation of $x$ (resp.\ $y$) in $h_d$
(resp.\ $t$), and $\psi$ is a splitting morphism.
\end{propos}
\begin{proof}
From \cite[Example  8.3.7]{Ful} and \cite[Section~2.5]{EG}, we have an
exterior product ring homomorphism
\[
A^*_{\GL_n}(X) \otimes A^*_{\GL_n}(Y) \to A^*_{\GL_n}(X
\times Y)
\]
whenever $X$ and $Y$ are nonsingular varieties, and this homomorphism
is an isomorphism if one of the two nonsingular varieties is a
projective space. Consequently, $A^*_{\GL_n}  ( \Pro(W_d) \times
\Pro(E)  )$ is an $A^*_{\GL_n}$-module freely generated by the
set $\{ h_d^it^{j}\mid   0\leq i < N, 0 \leq j < n \}$. To define the
splitting morphism $\psi$, let $\gamma\in A^*_{\GL_n}  ( \Pro
(W_d) \times\Pro(E)  )$. There is a unique polynomial $Q(x,y)
\in A^*_{\GL_n}[x,y]$ whose degree in $x$ (resp.\ $y$) is less than
$N$ (resp.\ $n$) and $Q(h_d,t)=\gamma$. We define $\psi(\gamma
):=Q(x,y)$. It is again straightforward to prove that $\psi$ is a
splitting morphism.
\end{proof}

\begin{propos}\label{class.of.Z.tilde}
The ideal $i_* ( A^*_{\GL_n}(\wt{Z}) )$ is generated by
the class $[\widetilde{Z} ]_{\GL_n}$.
\end{propos}

\begin{proof}
Consider the following commutative diagram:
\[
\xymatrix{
\wt{Z} \ar[rr]^{i} \ar[drr]& &\Pro(W_d) \times\Pro(E) \ar
[d]^{\pr}\\
& & \Pro(E)
}
\]
First of all, notice that the subscheme $\wt{Z}$ is equivariant for
the action of $\GL_n$. Moreover, $\wt{Z}$ is a complete intersection
of $n$ equations, which are linear in the $W_d$ coordinates, and such
equations remain linearly independent on each fiber of $\pr$.
Therefore, we have that $\wt{Z}$ is the projectivization of an
equivariant subbundle of $W_d \times\Pro(E)$ over $\Pro(E)$.
Consequently, $A^*_{\GL_n}(\wt{Z})$ is generated by the set $\{
i^* ( h^j  )\mid   0 \leq j < N\}$ as $A^*_{\GL_n} (
\Pro(E)  )$-module. This means that $i_* (A^*_{\GL_n}(\wt
{Z}) )$ is generated by the set $\{ h^j [\wt{Z}]_{\GL_n}\mid   0
\leq j < N\}$ as a module and by $[\wt{Z}]_{\GL_n}$ as an ideal.
\end{proof}

\begin{propos} We have
\[
[ \widetilde{Z} ]_{\GL_n} = P_{\{ 1\}}(h + (d-1)t).
\]
\end{propos}

\begin{proof}
Since $\GL_n$ is special, by using Lemma \ref{algebraic.lemma}, part
I), we may perform our computations by restricting ourselves to $A^*_{T}$.

Recall that $\widetilde{Z}$ is the complete intersection of
hyperplanes given by polynomials
\[
F_{X_i} (X) = \sum_{v \in\N^n(d)} v_i
a_v X^{v-\widehat{i}}
\]
for $i=0,\dots, n-1$, where $\widehat{i}$ is the vector having 1 as
the $i$th entry and 0 everywhere else. We will call these hyperplanes
$F_i$. From Lemma \ref{hyperplane.T.classes} we have
\[
[ F_i ]_T = h + (d-1)t + l_i
\]
since $ \lambda_i^{-1} F_i = (\lambda_1, \dots, \lambda_n) \cdot F_i$.
We conclude by noticing that $[ \widetilde{Z} ]_{T}=\prod_{i=0}^{n-1}
 [ F_i  ]_T$.
\end{proof}

\begin{lemma}\label{lemma.point}
Let $\gamma$ be a class in $A^*_{\GL_n} ( \Pro(W_d) \times\Pro
(E)  )$, and let $Q(x,y)=\sum_{i=0}^{n-1}g_i(x)y^i$ be $\psi
(\gamma)$ as defined in Theorem \ref{prod.Wd.E}. Then we have the
identity $\pi_{1*}(\gamma)=g_{n-1}(h_d)$.
\end{lemma}

\begin{proof}
Since $\pi_{1*}: A^*_{\GL_n} ( \Pro(W_d) \times\Pro(E)
 ) \to A^*_{\GL_n} ( \Pro(W_d)  )$ is a homomorphism
of $A^*_{\GL_n}$-modules, it suffices to determine $\pi
_{1*}(h_d^it^j)$ with $i=0, \dots,n-1$ and $j=0, \dots, n-1$. So we
are reduced to the nonequivariant case. If $j < n-1$, then we would
have a positive dimensional fiber, and thus $\pi_{1*}(h_d^it^j)=0$. On
the other hand, $\pi_{1*}(h_d^it^{n-1})=h_d^i$.
\end{proof}

\begin{theorem}\label{generators.I.Z.tilde} Let $Q_{[d]}(x,y)=\sum_{i=1}^{n} \alpha_i(x)y^{n-i}$ be the polynomial $\psi ( [\wt
{Z}]  )$. We have
\[
I_{\wt{Z}} = ( \alpha_1 ( h_d ), \dots,
\alpha_n ( h_d ) ).
\]
\end{theorem}

\begin{proof}

As a preliminary remark, notice that we have the identity
\[
Q_{[d]}(x,y):=P_{\{ 1\}}(x + (d-1)y) - (-(d-1))^nP_{\{1\}}(-y).
\]

Let $J$ be the ideal in $A^*_{\GL_n} ( \Pro(W_{d})  )$
generated by the classes $\alpha_i ( h_d  )$. First, we
prove the inclusion $J \subseteq I_{\widetilde{Z}}$. It suffices to
show that, for all $i=1, \dots, n$, we have $\alpha_i(h_d) \in
I_{\widetilde{Z}}$. We apply induction on $i$. From Lemma \ref
{lemma.point} we have that $\alpha_1 ( h_d  )=\pi
_{1*} ( [\wt{Z}]  )$, and therefore $\alpha_1(h_d) \in
I_{\widetilde{Z}}$. Now, for every $i$ such that $1<i \leq n$, define
the class $B_i:= [\wt{Z}] \cdot t^{i-1}$, which clearly belongs to the
ideal generated by $ [\wt{Z}]$, and consequently $\pi_{1*}(B_i) \in
I_{\widetilde{Z}}$. We already know that
\[
B_i= Q_{[d]}(h_d,t)\cdot t^{i-1}=
\sum_{j=1}^{n} \alpha_j(h_d)t^{n+i
-1 - j}.
\]
We now split $B_i$ into the sum of three classes:
\[\begin{aligned}
B_i &= \sum_{j=1}^{i-1}
\alpha_j(h_d)t^{n+i -1 - j} + \alpha
_i(h_d)t^{n-1} + \sum
_{j=i+1}^{n} \alpha_j(h_d)t^{n+i -1 - j}
\\&=: \beta_i + \alpha_i(h_d)t^{n-1} +
\rho_i.\end{aligned}
\]
Since $\pi_{1*}$ is a homomorphism of $A^*_{\GL_n} ( \Pro
(W_{d})  )$-modules, we have that the class $\pi_{1*}(\beta_i)$
is in the ideal generated by the classes $\{ \alpha_j ( h_d
 ) \}_{j=1, \dots, i-1}$, which, by inductive hypothesis, is
contained in $I_{\widetilde{Z}}$. Moreover, by following the same
argument of Lemma \ref{lemma.point} we have $\pi_{1*}(\alpha
_i(h_d)t^{n-1}) = \alpha_i(h_d)$ and $\pi_{1*}(\rho_i)=0$. In
conclusion, we have
\[
\alpha_i(h_d) = \pi_{1*}(B_i) -
\pi_{1*}(\beta_i) \in I_{\widetilde{Z}}.
\]
Therefore, the classes $\alpha_i(h_d) \in I_{\widetilde{Z}}$ for all
$i=1, \dots, n$. Consequently, we have $J \subseteq I_{\widetilde{Z}}$.

On the other hand, let $\gamma$ be a class in $I_{\widetilde{Z}}$.
From Proposition \ref{class.of.Z.tilde} we have
\[
\gamma=\pi_{1*} ( B(h_d,t)\cdot Q_{[d]}(h_d,t)
)
\]
for some polynomial $B(x,y)\in A^*_{\GL_n}[x,y]$.
Now, it is straightforward to check that every $t$-coefficient of $\psi
 ( B(h_d,t) \cdot Q_{[d]}(h_d,t)  )$ is in the ideal
generated by the set $\{ \alpha_1(h_d), \dots, \alpha_n(h_d) \}$.
Therefore we have that $\gamma$ is in $J$ and $I_{\widetilde{Z}}
\subseteq J$.
\end{proof}

\begin{rmk}\label{rational.coefficients} Throughout this paper, we are
interested in integral coefficients. However, it is worth noticing
that, since the map $\widetilde{Z} \to Z$ is birational, the
pushforward morphism
\[
\pi_{1*}: A^*_{\GL_n} ( \widetilde{Z} ) \otimes\Q\to
A^*_{\GL
_n} ( Z ) \otimes\Q
\]
is surjective. Consequently, we have
\[
i_* ( A^*_{\GL_n} ( Z ) ) \otimes\Q= ( \alpha_1(h_d),
\dots, \alpha_n(h_d) ).
\]
\end{rmk}

As we already mentioned, the ideal $I_{\widetilde{Z}}$ is not, in
general, the whole ideal $I_{Z}$. However, the ideal $I_Z$ has been
already determined for quadrics (see \cite{EF} and \cite{Pan}) and
effective divisors of the projective line (see \cite{EFh}). More
precisely, we have the following result.

\begin{theorem}\label{lower.cases}
If $d=2$ or $n=2$, then $I_{\widetilde{Z}}=I_{Z}$.
\end{theorem}

\begin{proof}
If $d=2$, then the classes $\alpha_i(h_2)$ of $I_{\widetilde{Z}}$ are
exactly the classes of degree $i$ of $Q_{[d]}(h_2,1)$, and they are
equal to the classes $\alpha_i$ of \cite[Proposition~13]{EF}.

If $n=2$, then $\alpha_1(h_d)=2(d-1)h_d - d(d-1)c_1$ and $\alpha
_2(h_d)=h_d^2 -c_1h_d-d(d-2)c_2$, and these correspond to the classes
$\alpha_{1,0}$ and $\alpha_{1,1}$ of \cite[Lemma~13]{EFh}, which by
\cite[Theorem~19]{EFh} generate the ideal $I_Z$.
\end{proof}

We conclude this section by showing that the polynomial $P_{[d]}(x)$ is
in the ideal $ ( \alpha_1(x), \dots, \alpha_n(x)  )$.

\begin{propos}\label{big.polynomial}
We have $P_{[d]}(x) \in ( \alpha_1(x), \dots, \alpha_n(x)
 )$.
\end{propos}

\begin{proof}
If $d=2$, then we have the statement implicitly from \cite[Proposition~13]{EF}.

If $d > 2$, then we consider the following identity:
\[
P_{\{d\}}(x) \cdot P_{\{d-1,1\}}(x)=\prod
_{i=1}^{n}Q_{[d]}(x, l_i),
\]
where the polynomial $Q_{[d]}(x, y)$ is defined as in Theorem \ref
{generators.I.Z.tilde}. Since the polynomial $P_{[d]} ( x
)$ is a multiple of $P_{\{d\}}(x) \cdot P_{\{d-1,1\}}(x)$, we have that
$P_{[d]} ( x  )$ is a multiple of $\prod_{i=1}^{n}Q_{[d]}(x, l_i)$. Thus, it suffices to show that, in
$A^*_{\GL_n}[x]$, we have
\[
\prod_{i=1}^{n}Q_{[d]}(x,
l_i) \in ( \alpha_1(x), \dots, \alpha_n(x) ).
\]
This is clearly true in $A^*_T[x]$, and we conclude by using Lemma \ref
{algebraic.lemma}, part II).
\end{proof}

\section{The Case of Plane Cubics}\label{plane.cubics}

We consider now the particular case of plane cubics, namely the case
$n=3$ and $d=3$. In particular, we will give the following minimal set
of generators for the ideal $I_Z$:
\begin{equation}
\label{main.result} I_{Z}= ( \alpha_1, \alpha_2,
\alpha_3, \delta_{2} ),
\end{equation}
where, for simplicity, $\alpha_i:=\alpha_i(h_3)$, and $\delta_{2}$
is the class of the locus of cubics that are the unions of a line and a
conic. This is also the first case where $I_{\widetilde{Z}}\neq I_Z$.

First of all, we write explicitly the classes $\alpha_i$ by using
Theorem~\ref{generators.I.Z.tilde}:
\begin{align}
\label{alpha.classes}
\begin{split} \alpha_1&=12(h_3
- c_1),
\\
\alpha_2&= 6h_3^2 - 4h_3c_1
- 6 c_2,
\\
\alpha_3&= h_3^3 - h_3^2c_1
+ h_3c_2-9c_3. \end{split} %
\end{align}

\subsection{Stratification}\label{stratification}

\begin{defi}\label{strata}
We consider the following loci in $Z$:
\begin{itemize}
\item$Z_1$ is the locus of reduced and irreducible singular cubics
(with exactly one singular point);
\item$Z_2$ is the locus of cubics that are unions of a smooth conic
and a line with two distinct intersection points;
\item$Z_3$ is the union of two components:
\begin{itemize}
\item[] $Z_{(3,1)}$ is the locus of cubics that are unions of a smooth
conic and a line tangent to the conic;
\item[] $Z_{(3,2)}$ is the locus of cubics that are unions of three
distinct lines with three distinct intersection points;
\end{itemize}
\item$Z_4$ is the locus of cubics that are unions of three distinct
lines passing through the same point;
\item$Z_5$ is the locus of cubics that are unions of a double line and
a single distinct line;
\item$Z_7$ is the locus of triple lines.
\end{itemize}
\end{defi}

\begin{rmk}
All $Z_i$ are smooth and locally closed in $\Pro ( W_3  )$.
Furthermore, we have that $Z$ is the closure of $Z_1$ in $\Pro (
W_3  )$. We also observe that $\ov{Z}_{(3,1)}\cap\ov
{Z}_{(3,2)}=\ov{Z_4}$. Moreover, we have chosen the indexes in such a
way that $Z_i$ has codimension $i$ in $\Pro(W_3)$. Notice that there
is a gap in codimension 6. Finally, we observe that all these loci are
invariant for the action of $\GL_3$.
\end{rmk}

Such a stratification of the singular locus $Z$ is equipped with a
natural partial ordering given by
\[
Z_i \leq Z_j \longleftrightarrow\overline{Z}_i
\subseteq\overline{Z}_j.
\]

So we can represent such a stratification with a digraph. On the left
column, we write the codimension of the corresponding strata in $\Pro(W_3)$.
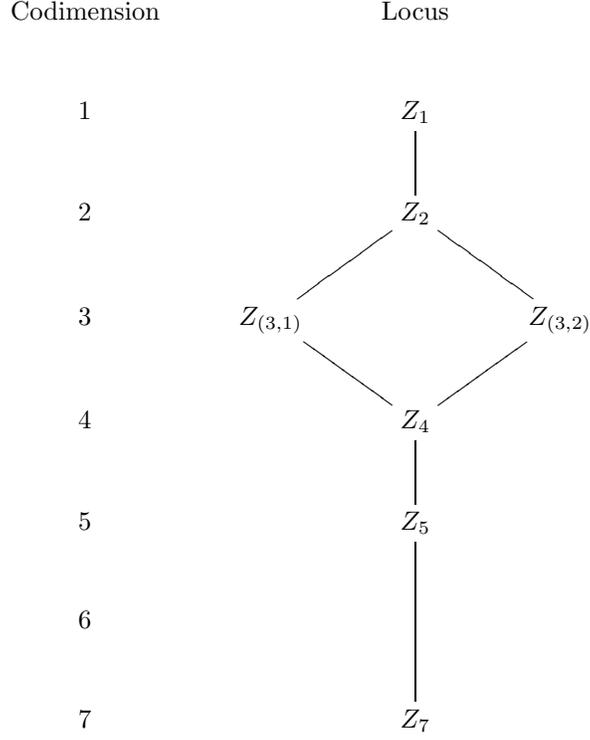
\[
\xymatrix{
\text{Codimension} & & \text{Locus} &\\
1 & & Z_{1} \ar@{-}[d] & \\
2 & & Z_{2} \ar@{-}[dr] \ar@{-}[dl]& \\
3 & Z_{(3,1)} \ar@{-}[dr] & & Z_{(3,2)} \ar@{-}[dl] \\
4 & & Z_{4}\ar@{-}[d]&\\
5 & & Z_5 \ar@{-}[dd]&\\
6 & & & \\
7 & & Z_7 &\\
}
\]

\begin{defi} We define the {\it topological classes} corresponding to
the above loci: $\delta_{i}:= [ \ov{Z}_i ]$.
\end{defi}

\begin{rmk}
We have $[Z]=\delta_1=\alpha_1$.
\end{rmk}

\begin{defi}
We define the following maps (see Section~\ref{universal-locus}):
\begin{itemize}
\item$\pi_2:= \pi_{\{ 1,2 \}}:\Pro ( W_1  ) \times
\Pro ( W_2  ) \to\Pro ( W_3  )$, where $(f,g)
\mapsto f \cdot g$;
\item$\pi_3:= \pi_{ \{ 1,1,1 \} }:\Pro ( W_1
)^{\times3} \to\Pro ( W_3  )$, where
$(f,g,h) \mapsto f \cdot g \cdot h$.
\end{itemize}
\end{defi}

Recall that we have already defined the map $\pi_1:\wt{Z} \to
Z$ (see Section~\ref{universal-locus}). We notice that, for $i=1, 2$,
we have $\im ( \pi_i  )= \ov{Z}_i$ and $\im ( \pi
_3  )= \ov{Z}_{(3,2)}$. Also, all these maps are invariant for
the action of $\GL_3$. Moreover, $\pi_1$ and $\pi_2$ are birational
to their images.

\subsection{Basic Principle of Proof} \label{basic.principle}

The proof of identity (\ref{main.result}) is split into several steps.
We rely on the following basic principle. Let
\[
Y_n \subset Y_{n-1} \subset\cdots\subset Y_1
\subset X
\]
be a sequence of closed $G$-subspaces of a smooth $G$-space $X$ ordered
by inclusion, and let $I$ be an ideal in $A^*_G(X)$. We call $i$ all
the closed inclusions $Y_k \to X$ and $Y_{k} \setminus Y_{k+1} \to
X \setminus Y_{k+1}$, whereas we denote by $j$ the open inclusions
$X \setminus Y_{k} \to X$.

Then, to show the inclusion $i_* ( A^G_*(Y_1) ) \subset I$,
it suffices to show that\break  $i_*  ( A^G_*(Y_n)  ) \subset I$
and that, for all $k=1, \dots, n-1$, we have $i_*  ( A^G_* ( Y_k
\setminus Y_{k+1})  ) \subset j^*(I) \subset A^*_G  ( X
\setminus Y_{k+1}  )$.

In our case, we first show the inclusion $i_*  ( A^{\GL_3}_*
 ( \ov{Z}_{(3,1)} ) ) \subseteq I_{\wt{Z}}$. This\break
also implies that $i_*  ( A^{\GL_3}_*  ( \ov{Z}_4
) ) \subseteq I_{\wt{Z}}$. Then we can prove the inclusion\break  $i_*
 ( A^{\GL_3}_*  ( \ov{Z}_{(3,2)} ) ) \subseteq
I_{\wt{Z}}$ by restricting ourselves to the open set $\Pro(W_3)
\setminus\ov{Z}_4$. At this point, we have that $i_*  (
A^{\GL_3}_*  ( \ov{Z}_{3} ) ) \subseteq I_{\wt
{Z}}$, and we can show the inclusion $i_*  ( A^{\GL_3}_*  (
\ov{Z}_{2} ) ) \subseteq(\alpha_1, \alpha_2, \alpha_3,
\delta_{2})$ by restricting ourselves to the open set $\Pro(W_3)
\setminus\ov{Z}_3$, which will conclude the proof of
identity~(\ref{main.result}).

\subsection{The Ideal $i_*  ( A^{\GL_3}_*  ( \ov
{Z}_{(3,1)} ) )$ Is Contained in $I_{\wt{Z}}$}\label{Z31}

\begin{propos}\label{first.stratum}
Let us define $\partial Z_4 := \ov{Z}_4 \setminus Z_4$. We have
the inclusion
\[
i_* ( A_*^{\GL_3} (\partial Z_4 ) ) \subseteq
I_{\wt{Z}}.
\]
\end{propos}

\begin{proof}
The algebraic set $\partial Z_4$ can be stratified as $Z_5 \sqcup Z_7$.

Notice that $\partial Z_4=\ov{Z}_5$. We also define $\wt{Z}_5=\pi
_{1}^{-1}(Z_5)$ and $\wt{Z}_7=\pi_{1}^{-1}(Z_7)$.

We refer to the following commutative diagram of $\GL_3$-equivariant maps:
\[
\xymatrix{
\wt{Z}_7 \ar[r]^{i} \ar[d]^{\pi_1} & \ov{\wt{Z}}_5 \ar[r]^{i}
\ar[d]^{\pi_1} & \wt{Z} \ar[d]^{\pi_1} \\
Z_7 \ar[r]^{i} & \ov{Z}_5 \ar[r]^-{i} & \Pro(W_3)
}
\]

Because of commutativity of this diagram and the basic principle
explained in Section~\ref{basic.principle}, it suffices to prove
that the homomorphisms $\pi_{1*}: A^*_{\GL3} (\wt{Z}_7) \to
A^*_{\GL3} (Z_7)$ and $\pi_{1*}: A^*_{\GL3} ( \wt{Z}_5) \to
A^*_{\GL3} (Z_5)$ are surjective. Consider the following commutative diagram:
\[
\xymatrix{
\wt{V} \ar[r]^{\wt{\psi}} \ar[d]^{\pi_1}& \wt{Z}_7 \ar[d]^{\pi
_1} \\
\Pro(W_1) \ar[r]^-{\psi} & Z_7
}
\]
where $\wt{V}$ is the incidence variety if $\Pro(W_1) \times\Pro
(E)$, and the map $\psi$ (resp.\ $\wt{\psi}$) sends $[l]$ (resp.\
$([l], P)$) to $[l^3]$ (resp.\ $([l^3], P)$). In particular, $\pi
_{1}:\wt{V} \to\Pro(W_1)$ is a projective bundle, and
therefore $\pi_{1*}: A^{*}_{\GL_3}(\wt{V}) \to A^{*}_{\GL
_3}(\Pro(W_1)) $ is surjective. Moreover, the maps $\psi$ and $\wt
{\psi}$ are geometrically bijective, and a straightforward computation
shows that their induced Jacobian maps are injective (here we use the
fact that $\Char(k) > 3$). Therefore, the maps $\psi$ and $\wt{\psi
}$ are isomorphisms, and consequently $\pi_{1*}: A^*_{\GL3} (\wt
{Z}_7) \to A^*_{\GL3} (Z_7)$ is surjective.

To show that $\pi_{1*}: A^*_{\GL3} ( \wt{Z}_5) \to A^*_{\GL3}
(Z_5)$ is surjective, we apply a similar argument to the diagram
\[
\xymatrix{
\wt{V} \ar[rr]^{\wt{\psi}} \ar[d]^{\pi_1}& &\wt{Z}_5 \ar
[d]^{\pi_1} \\
(\Pro(W_1) \times\Pro(W_1) \setminus D) \ar[rr]^-{\psi} & & Z_5
}
\]
where $D$ is the diagonal, $\wt{V}$ is the incidence variety (with
respect to the first component) in $(\Pro(W_1) \times\Pro(W_1)
\setminus D) \times\Pro(E)$, and the map $\psi$ (resp.\ $\wt
{\psi}$) sends $[l],[w]$ (resp.\ $([l],[w], P)$) to $[l^2w]$ (resp.\
$([l^2w], P)$).
\end{proof}

\begin{propos}\label{envelope.3.1}
The restriction map
\[
\wt{V}:= \wt{Z} \rrvert _{Z_{(3,1)} \cup Z_4} \xrightarrow {\pi_1}
Z_{(3,1)} \cup Z_4
\]
is an equivariant Chow envelope.
\end{propos}

\begin{proof}

Recall that an equivariant Chow envelope of a $G$-scheme $X$ is a
proper $G$-equivariant morphism $f:\wt{X} \to X$ such that, for every
$G$-invariant subvariety $W \subset X$, there is a $G$-invariant
subvariety $\wt{W} \subset f^{-1}(W)$ whose restriction morphism
$f:\wt{W} \to W$ is birational.

We already know that the map
\[
\wt{V}:= \wt{Z} \rrvert _{Z_{(3,1)} \cup Z_4} \xrightarrow {\pi_1}
Z_{(3,1)} \cup Z_4
\]
is proper and $\GL_3$-equivariant.

It suffices to show the statement for each of the two components
$Z_{(3,1)}$ and $Z_4$. Since the proofs are very similar, we only show
the case of $Z_{(3,1)}$.

Let $W$ be a $\GL_{3}$-invariant subvariety in $Z_{(3,1)}$. Let
$\omega$ be the generic point of~$W$. The point $\omega$ is
represented by a cubic form $f$ in $K[X_0,X_1,X_2]$ for some extension
$k \subset K$. By definition $f$ is the product of a linear and a
quadratic form with only one singular point. Therefore there is a
unique $K$-valued point $\wt{\omega}$ in $\wt{Z}_{(3,1)}$ mapping to
$\omega$. To conclude, define $\wt{W} := \ov{\wt{\omega}}$;
because of the uniqueness of the rational point, $\wt{W}$ is in fact
$\GL_{3}$-invariant.
\end{proof}

\begin{corol} \label{tangent}
We have the inclusion
\[
i_* ( A_*^{\GL_3} ( \ov{Z}_{(3,1)} ) ) \subseteq
I_{\wt{Z}}.
\]
\end{corol}
\begin{proof}
Consider the following commutative diagram of proper maps:
\[
\xymatrix{
\wt{Z} |_{ \ov{Z}_{(3,1)}} \ar[r]^-{i} \ar[d]_{\pi_1} & \wt Z \ar
[d]^{\pi_1}\\
\ov{Z}_{(3,1)} \ar[r]^{i} & \Pro(W_3)
}
\]
Since by Proposition \ref{envelope.3.1} the map $\pi_1:\wt{V}
\to Z_{(3,1)} \cup Z_4$ is an equivariant Chow envelope, we have from
\cite[Lemma 3]{EG} and \cite[Lemma 18.3(6)]{Ful} that the pushforward
$\pi_{1*}: A^{\GL_3}_*  ( \wt{V} ) \to A^{\GL_3}_*
 ( Z_{(3,1)} \cup Z_4  )$ is surjective.
By the commutativity of the diagram we have
\[
i_* ( A_*^{\GL_3} ( Z_{(3,1)} \cup Z_4 ) ) \subseteq
 I_{\wt{Z}} \rrvert _{Z_{(3,1)} \cup Z_4}.
\]

We conclude by recalling that $i_*  ( A_*^{\GL_3}  (\partial
Z_4  ) ) \subseteq I_{\wt{Z}}$ by Proposition \ref
{first.stratum} and by using the basic principle of Section~\ref{basic.principle}.
\end{proof}

\subsection{The Ideal $i_*  ( A^{\GL_3}_*  ( \ov
{Z}_{(3,2)} ) )$ Is Contained in $I_{\wt{Z}}$}\label{Z32}

Let us consider the product map
\[
\Pro(W_1)^{\times3} \xrightarrow{\pi_3}
\ov{Z}_{(3,2)}.
\]
We call $\xi_1$, $\xi_2$, and $\xi_3$ the three hyperplane classes
corresponding to the pullback of hyperplane classes through the three
different projections from $\Pro(W_1)^{\times3}$ to $\Pro(W_1)$.
Arguing as in Proposition \ref{prod.Wd.E}, we have a splitting exact
sequence of $A^*_{\GL_3}$-modules
\[
\xymatrix@R=10pt{
0 \ar[r] &  (P_{\{1\}}(y_1), P_{\{1\}}(y_2), P_{\{1\}}(y_3)
 ) \\\ar[r] & A^*_{\GL_3}[y_1,y_2,y_3] \ar[rr]_{\ev_{(\xi
_1,\xi_2,\xi_3)}} & &  A^*_{\GL_3}  (\Pro(W_1)^{\times3}
 ) \ar[r] \ar@/_1pc/[ll]_{\psi} & 0.
}
\]

To prove the inclusion $i_*  ( A^{\GL_3}_*  ( \ov
{Z}_{(3,2)}  ) ) \subseteq I_{\wt{Z}}$, using the explicit
localization theorem (Theorem \ref{explicit.localization}), we first
show that $i_* ( \delta_{(3,2)}  ) \in I_{\wt{Z}}$ (where
we recall that $\delta_{(3,2)}:= [\ov{Z}_{(3,2)}]$).

\begin{propos}\label{class.delta.3.2}
We have the identity
\begin{equation}
\label{delta_3.2} \delta_{(3,2)}= ( ( h_3-c_1
)^2 + c_2 )\alpha _1 - c_1
\alpha_2 + 3\alpha_3.
\end{equation}
\end{propos}

\begin{proof} We refer to Section~\ref{universal-locus}. Since $\delta
_{(3,2)}=\delta_{ \{ 1,1,1\} }$, we evaluate formula (\ref
{delta.classes.formula}) for $\mu=\{ 1,1,1\}$ and $d=3$.

As preliminary computations, we get
\begin{eqnarray*}
c^T_{\mathrm{top}} ( T_{Q_{(1,0,0)}} \Pro (W_1 ) )
&=& (l_2 - l_1) (l_3 - l_1),
\\
c^T_{\mathrm{top}} ( T_{Q_{(0,1,0)}} \Pro (W_1 ) )
&=& (l_1 - l_2) (l_3 - l_2),
\\
c^T_{\mathrm{top}} ( T_{Q_{(0,0,1)}} \Pro (W_1 ) )
&=& (l_1 - l_3) (l_2 - l_3).
\end{eqnarray*}

Now, straightforward computations show the relation
\[
\delta_{(3,2)}= 15h_3^3 - 45c_1h_3^2
+ ( 40c_1^2 + 15 c_2 ) h_3
-12c_1^3 - 6c_1c_2
-27c_3
\]
and, consequently, identity (\ref{delta_3.2}).
\end{proof}

\begin{defi}
Let $X$ be a $G$-space, and let $\Gamma$ be a finite group acting
(properly) on $X$ such that the action of $\Gamma$ commutes with the
action of $G$. We say that two classes $\gamma_1, \gamma_2 \in
A^*_{G}(X)$ are $\Gamma$-equivalent if, for some $f \in\Gamma$, we
have $f_* ( \gamma_1  )= \gamma_2$.
\end{defi}

\begin{propos}\label{push.forward.pi3}
We have the inclusion
\[
\pi_{3*} ( A_*^{\GL_3} ( \Pro(W_1)^{\times3}
) ) \subset I_{\wt{Z}}.
\]
\end{propos}

\begin{proof}
The free $A^*_{\GL_3}$-module $\psi ( A^*_{\GL_3}  ( \Pro
(W_1)^{\times3}  )  )$ is generated by monomials
$y_1^{v_1} y_2 ^{v_2} y_3^{v_3}$ such that every nonnegative integer
$v_i$ is less than 3. Therefore, it suffices to consider the
pushforward of the classes $\xi_1^{v_1} \xi_2^{v_2} \xi_3^{v_3}$
where each $v_i$ is either 0, 1, or 2. Moreover, $\pi_3^*
(h_3 ) = \xi_1 + \xi_2 + \xi_3$, and applying the push--pull
formula, we have that $\pi_{3*}  ( A_*^{\GL_3}  ( \Pro
(W_1)^{\times3}  )  )$ is generated by the classes $\pi
_{3*} ( \xi_1^{v_2}\xi_2^{v_2} )$. Now, we notice that the
morphism $\pi_3$ is $S_3$-equivariant, where the action on $\Pro
(W_1)^{\times3}$ is permuting the three components, and the action on
$\ov{Z}_{(3,2)}$ is trivial. Therefore, for every $f \in S_3$, we have
the commutative diagram
\[
\xymatrix{
\Pro(W_1)^{\times3} \ar[rr]^{f} \ar[dr]_{\pi_3}& & \Pro
(W_1)^{\times3} \ar[dl]^{\pi_3}\\
& \ov{Z}_{(3,2)} &
}
\]
This means, in particular, that if $\gamma:=\xi_1^{v_1}\xi_2^{v_2}$
and $\gamma':=\xi_1^{v'_1}\xi_2^{v'_2}$ are $S_3$-equivalent, then
$\pi_3  ( \gamma )=\pi_3  ( \gamma'  )$.
Consequently, we can see that $\pi_{3*}  ( A_*^{\GL_3}  (
\Pro(W_1)^{\times3}  )  )$ is generated by the classes
$\pi_{3*} ( \xi_1^{v_2}\xi_2^{v_2} )$ with $2 \geq v_1
\geq v_2 \geq0$. We consider each of the six cases separately.
\begin{itemize}

\item$\pi_{3*}  ( 1  ) \in I_{\wt{Z}}$.

Since $\pi_{3*}  ( 1  )=6\delta_{(3,2)}$, this case is
covered in Proposition \ref{class.delta.3.2}.

\item$\pi_{3*}  ( \xi_1  ) \in I_{\wt{Z}}$.

Consider the identities $\pi_3^* (h_3 ) = \xi_1 + \xi_2 +
\xi_3$ and $\pi_{3*}(\xi_1)=\pi_{3*}(\xi_2)=\pi_{3*}(\xi_3)$. By
applying the push--pull formula we get $\pi_{3*}  ( \xi_1
)=2h_3\delta_{(3,2)}$.

\item$\pi_{3*}  ( \xi_1 \xi_2  ) \in I_{\wt{Z}}$.

By Lemma \ref{algebraic.lemma} we can restrict our computations to
$A^*_T$. Let us define $S_1 \subset\Pro(W_1)^{\times3}$ as the locus
where the first two lines pass through $[1,0,0]$. We apply Lemma \ref
{hyperplane.T.classes} to get
\[
[ S_1 ] = (\xi_1 + l_1) (\xi_2 +
l_1) = \xi_1\xi_2 + l_1 (
\xi_1 + \xi_2) + l_1^2 \in
A^*_T ( \Pro(W_1)^{\times3} ).
\]
We have the following commutative diagram:
\begin{equation}
\label{split.diagram} \xymatrix{
& \Pro(W_3) \times\Pro(E) \ar[dr]^{\pi_1} & \\
\Pro(W_1)^{\times3} \ar[rr]^{\pi_3} \ar[ur]^{\wt{\pi}_3} & &
\Pro(W_3)
}
\end{equation}
where $\wt{\pi}_3$ maps $(f,g,h)$ to $(fgh, [1,0,0])$. By definition
we have that $\wt{\pi}_3(S_1) \subset\wt{Z}$; therefore, since the
diagram is commutative, we have $\pi_{3*}( [ S_1  ]) \in
I_{\wt{Z}}A^*_T  ( \Pro(W_3)  )$. More explicitly, we have
\[
\pi_{3*} ( \xi_1\xi_2 ) + 2l_1
\pi_{3*} ( \xi_1 ) + l_1^2
\delta_{(3,2)} \in I_{\wt{Z}}A^*_T ( \Pro
(W_3) ).
\]
Therefore $\pi_{3*} ( \xi_1\xi_2  ) \in I_{\wt{Z}}A^*_T
 ( \Pro(W_3)  )$. From the argument of Remark \ref
{remark.algebraic.lemma} we have $\pi_{3*} ( \xi_1\xi_2
) \in I_{\wt{Z}}$.

\item$\pi_{3*}  ( \xi_1^2  ) \in I_{\wt{Z}}$.

Consider the identities
\begin{eqnarray*}
 \pi_3^*(h_3^2) &=& \xi_1^2
+ \xi_2^2 + \xi_3^2 + 2(
\xi_1\xi_2 + \xi_1\xi_3 +
\xi_2\xi_3),
\\
 \pi_{3*} (\xi_1^2) &=& \pi_{3*} (
\xi_2^2) = \pi_{3*} (\xi_3^2),
\\
 \pi_{3*} (\xi_1\xi_2) &=&
\pi_{3*} (\xi_1\xi_3) = \pi_{3*} (
\xi _2\xi_3).
\end{eqnarray*}
By applying the push--pull formula we get $\pi_{3*}  ( \xi_1^2
 )=2h_3^2 \delta_{(3,2)} - 2 \pi_{3*}(\xi_1 \xi_2)$.

\item$\pi_{3*}  ( \xi_1^2 \xi_2  ) \in I_{\wt{Z}}$.

We argue exactly as in the proof of $\pi_{3*}  ( \xi_1 \xi_2
 ) \in I_{\wt{Z}}$. In this case, we choose $S_1$ to be the
locus where the first line passes through $[1,0,0]$ and $[0,1,0]$,
whereas the second line passes through $[1,0,0]$. We get
\[
[S_1] = (\xi_1 + l_1) (\xi_1 +
l_2) (\xi_2 + l_1).
\]
Again, the map $\pi_3$ factors through $\wt{\pi}_3$, and a simple
computation shows that $\pi_{3*}  ( \xi_1^2 \xi_2  ) \in
I_{\wt{Z}}$.

\item$\pi_{3*}  ( \xi_1^2 \xi^2_2  ) \in I_{\wt{Z}}$.

We argue again as in the proof of $\pi_{3*}  ( \xi_1 \xi_2
 ) \in I_{\wt{Z}}$. In this case, we choose $S_1$ to be the
locus where the first line passes through $[1,0,0]$ and $[0,1,0]$,
whereas the second line passes through $[1,0,0]$ and $[0,0,1]$. We get
\[
[S_1] = (\xi_1 + l_1) (\xi_1 +
l_2) (\xi_2 + l_1) (\xi_2 +
l_3).
\]
As before, the map $\pi_3$ factors through $\wt{\pi}_3$, and a
simple computation shows that $\pi_{3*}  ( \xi_1^2 \xi^2_2
 ) \in I_{\wt{Z}}$.
\end{itemize}
%
\end{proof}

Our next goal is to prove Corollary \ref{gamma.in.IZtilde}. Since we
already know that\break  $i_* ( A^{\GL_3}_*  ( \ov{Z}_4
) ) \subset A^*_{\GL_3}  ( \Pro(W_3)  )$ is
contained in $I_{\wt{Z}}$, we can restrict ourselves to classes in
$A^{\GL_3}_*  ( \ov{Z}_{(3,2)} )$ up to classes in
$i_* ( A^{\GL_3}_*  ( \ov{Z}_4 ) ) \subset
A^*_{\GL_3}  (\ov{Z}_{(3,2)}  )$.

We will need the following fact: for any class $\gamma\in A^{\GL_3}_*
 ( \ov{Z}_{(3,2)} )$, there exist two classes $\gamma' \in
i_* A^{\GL_3}_*  ( \ov{Z}_{4}  ) \subset A^{\GL_3}_*
 ( \ov{Z}_{(3,2)} )$ and $\ov{\gamma} \in A^*_{\GL_3}
 ( \Pro(W_1)^{\times3}  )^{S_3}$ such that $6(\gamma-
\gamma') = \pi_{3*} \ov{\gamma}$. This is a particular case of the
following lemma.

\begin{lemma}\label{pushforward.symmetric}
Let $G$ be an affine algebraic group, and let $\Gamma$ be a finite
group. Suppose that the following conditions are satisfied:
\begin{enumerate}[(3)]

\item[(1)]$G$ and $\Gamma$ act on an algebraic variety $X$, and the two
actions commute.

\item[(2)]$G$ also acts on an algebraic variety $Y$, and $f: X \to Y$
is a proper $G$-equivariant and $\Gamma$-invariant morphism.

\item[(3)]$V \subseteq Y$ is an open $G$-invariant subscheme such that if
$U = f^{-1}(V)$, and the restriction $f\mid_{U}: U \to V$ is
finite and flat with constant degree~$d$.
\end{enumerate}

Set $Z = Y \setminus V$ and call $i: Z \to Y$ the embedding.

Then, for any class $\gamma\in A^{G}_{*}(Y)$, we can find $\gamma'
\in i_*(A^{G}_*(Z))$ and $\overline{\gamma} \in A^{G}_*(X)^{\Gamma}$
such that
\[
d(\gamma- \gamma') = f_{*} \ov{\gamma}.
\]
\end{lemma}

\begin{proof}
Let $E$ be a representation of $G$, and let $D \subset E$ be an
equivariant open subset of $E$ such that $G$ acts freely on $D$.
Replacing $X$ with $\frac{X \times D}{G}$, and the same for $Y$ and
$V$, and assuming that the codimension of $E \setminus D$ in $E$ is
sufficiently large, we are reduced to the nonequivariant case: in other
words, we can assume that $G$ is trivial.

Let $\gamma$ be a class in $A_{k}(Y)$. Write
\[
\gamma= \sum_{i} a_i [ A_i
] + \sum_{j} b_{j}[B_{j}],
\]
where $a_{i}$ and $b_{j}$ are integers, $A_{i}$ are subvarieties of $Y$
not contained in $Z$, whereas $B_{j}$ are subvarieties of $Z$. We set
$\gamma' = \sum_{j} b_{j}[B_{j}]$, so that $\gamma- \gamma' = \sum_{i} a_i  [ A_i  ]$.

For each $i$, denote by $A'_{i}$ the scheme-theoretic inverse image of
$A_{i} \cap V$ in $U$ and by $\overline{A}_{i}$ the scheme-theoretic
closure of $A'_{i}$ in $X$. Each $\overline{A}_{i}$ is a purely
$k$-dimensional $\Gamma$-invariant subvariety of $X$. Set $\overline
{\gamma} = \sum_{i} a_i[\overline{A}_{i}]$. These $\gamma'$ and
$\overline{\gamma}$ satisfy the conditions of the statement.
\end{proof}

In the proof of Proposition \ref{2gamma}, we will also use the
following lemma.

\begin{lemma}\label{3IZtilde}
Consider the following diagram:
\[
\xymatrix{
& \Pro(W_1)^{\times3} \ar[d]^{\pi_3}\\
\ov{Z}_{(3,2)} \ar[r]^{i} & \Pro(W_3)
}
\]
Let $\gamma$ be a class in $A^{\GL_3}_*  ( \ov{Z}_{(3,2)}
 )$, and let $\ov{\gamma}$ be a class in $A^*_{\GL_3}  (
\Pro(W_1)^{\times3}  )^{S_3}$ such that $6 i_*(\gamma)=\pi
_{3*}(\ov{\gamma})$. Then $\pi_{3*}  ( \ov{\gamma}  )
\in3 I_{\wt{Z}}$.
\end{lemma}

\begin{proof}
First of all, we show that if the degree of $\ov{\gamma} \in A^*_{\GL
_3}  ( \Pro(W_1)^{\times3}  )^{S_3}$ is less than 5 in
$\xi_1, \xi_2, \xi_3$, then $\pi_{3*}(\ov{\gamma}) \in3 I_{\wt
{Z}}$. Notice that $ ( A^*_{\GL_3} ( \Pro(W_1)^{\times3}
 )  )^{S_3} $ is generated by the symmetrization of the
classes $\xi^{v} :=\xi_1^{v_1}\xi_2^{v_2}\xi_3^{v_3}$ for some
integral vector $v=(v_1,v_2, v_3)$ such that $2 \geq v_1 \geq v_2 \geq
v_3 \geq0$. Let $\wh{\xi^v}$ be the symmetrization of $\xi^v$.
Arguing as in the proof of Proposition \ref{push.forward.pi3}, we get
\[
\pi_{3*} ( \wh{\xi^{v}} ) = \# \text{orb} ( \xi
^{v} ) \pi_{3*} ( \xi^{v} ) \in\# \text{orb} (
\xi^{v} ) I_{\wt{Z}},
\]
where $\# \text{orb}  ( \xi^{v} )$ is the cardinality of
the $S_3$-orbit of $\xi^v$. If the entries of $v$ are all different
(namely the case $v=(2,1,0)$), then $\# \text{orb}  ( \xi
^{v} )=6$. On the other hand, if two entries of $v$ are equal and
one different from the other two, then $\# \text{orb}  ( \xi
^{v} )=3$. Therefore we are reduced to check the two cases $\pi
_{3*} (1)$ and $\pi_{3*} (\xi_1\xi_2\xi_3)$.
\begin{itemize}

\item$\pi_{3*} (1) \in3I_{\wt{Z}}$.

Again, since $\pi_{3*}  ( 1  )=6\delta_{(3,2)}$, this case
is covered in Proposition \ref{class.delta.3.2}.

\item$\pi_{3*} (\xi_1\xi_2\xi_3) \in3I_{\wt{Z}}$.

Using Lemma \ref{algebraic.lemma}, we can restrict our computations to
$A^*_T$. Let us define $S_1 \subset\Pro(W_1)^{\times3}$ as the locus
where all three lines pass through $[1,0,0]$. We apply Lemma \ref
{hyperplane.T.classes} to get
\[\begin{aligned}
{[ S_1 ]}& = (\xi_1 + l_1) (\xi_2 +
l_1) (\xi_3 + l_1) \\&= \xi _1
\xi_2 \xi_3 + l_1 (\xi_1
\xi_2 + \xi_1 \xi_3 + \xi_2
\xi_3) + l_1^2(\xi_1 +
\xi_2 + \xi_3) + l_1^3.\end{aligned}
\]

Now, with reference to diagram (\ref{split.diagram}), we have that the
map $\wt{\pi}_3: S_1 \to\wt{Z}$ is generically $6:1$ on its image,
and therefore $\pi_{3*}[S_1] = 6 \beta$ where $\beta$ is a class in
$I_{\wt{Z}}$. Therefore we have
\[
\pi_{3*} (\xi_1 \xi_2 \xi_3) = 6
\beta- l_1 \pi_{3*} (\xi_1 \xi _2 +
\xi_1 \xi_3 + \xi_2 \xi_3) -
l_1^2\pi_{3*}(\xi_1 +
\xi_2 + \xi_3) - l_1^3
\delta_{(3,2)},
\]
and we know that the right-hand side is in $3I_{\wt{Z}}$.
\end{itemize}

Let us go back to prove Lemma \ref{3IZtilde}. Since $\ov{\gamma}$ is
a class in $A^*_{\GL_3}  ( \Pro(W_1)^{\times3}  )^{S_3}$,
we can write
\[
\ov{\gamma}= a \xi_1^2 \xi_2^2
\xi_3^2 + \mu,
\]
where $\mu$ has degree at most 5 in $\xi_1, \xi_2, \xi_3$, and $a
\in A^*_{\GL_3}$. Applying $\pi_{3*}$ to both sides, we get
\[
\pi_{3*}(\ov{\gamma}) = a \pi_{3*} \xi_1^2
\xi_2^2 \xi_3^2 + \pi
_{3*} \mu.
\]
We already know that $\pi_{3*} \mu\in3I_{\wt{Z}}$. On the other
hand, by hypothesis we have $6 i_*(\gamma)=\pi_{3*}(\ov{\gamma})$.
Consequently, we must have that $ a \pi_{3*} \xi_1^2 \xi_2^2 \xi
_3^2 \in(3)$. Now, since the class $\pi_{3*} \xi_1^2 \xi_2^2 \xi
_3^2$ is the pushforward of a dimension 0 class, it must be a homogeneous
class of degree 9 in $A^*_{\GL_3}(\Pro(W_4))$. Moreover, in the
nonequivariant case, we have $\pi_{3*} \xi_1^2 \xi_2^2 \xi
_3^2=h_3^9$, and therefore, in the equivariant case, $\pi_{3*} \xi
_1^2 \xi_2^2 \xi_3^2$ can be written as a monic polynomial in $h_3$
with coefficients in $A^*_{\GL_3}$. Since $A^*_{\GL_3}(\Pro(W_3))$
is a free $A^*_{\GL_3}$-module with basis $1, h_3, \dots, h_3^9$, we
must have $a \in(3)$. On the other hand, we know already that $\pi
_{3*} \xi_1^2 \xi_2^2 \xi_3^2 \in I_{\wt{Z}}$, and consequently $a
\pi_{3*} \xi_1^2 \xi_2^2 \xi_3^2 \in3 I_{\wt{Z}}$ and $\pi
_{3*}(\ov{\gamma}) \in3 I_{\wt{Z}}$.
\end{proof}

\begin{propos}\label{2gamma}
Let $\gamma$ be a class in $A^{\GL_3}_*  ( \ov{Z}_{(3,2)}
 )$. Then $2 i_*(\gamma) \in I_{\wt{Z}}$.
\end{propos}

\begin{proof}
Let $\gamma$ be a class in $A^{\GL_3}_*  ( \ov{Z}_{(3,2)}
 )$. As a consequence of Lemma \ref{pushforward.symmetric},
there exist two classes $\gamma' \in i_* A^{\GL_3}_*  ( \ov
{Z}_{4}  ) \subset A^{\GL_3}_*  ( \ov{Z}_{(3,2)} )$
and $\ov{\gamma} \in A^*_{\GL_3}  ( \Pro(W_1)^{\times3}
 )^{S_3}$ such that
\begin{equation}
\label{equation.int.dom} 6 i_*(\gamma) = \pi_{3*} \ov{\gamma} + 6 i_*(
\gamma') \in A^*_{\GL
_3} ( \Pro(W_3) ).
\end{equation}
Now, from Lemma \ref{3IZtilde} we have that $\pi_{3*} \ov{\gamma}$
is three times a class $\beta$ in $I_{\wt{Z}}$. On the other hand, we
already know that $i_*(\gamma') \in I_{\wt{Z}}$. Since the ring
$A^*_{\GL_3}  ( \Pro(W_3)  )$ is an integral domain,
simplifying equation (\ref{equation.int.dom}), we get
\[
2 i_*(\gamma) = \beta+ 2 i_*(\gamma') \in I_{\wt{Z}}.
\]
%
\end{proof}

\begin{corol}\label{gamma.in.IZtilde}
Let $\gamma$ be a class in $A^{\GL_3}_*  ( \ov{Z}_{(3,2)}
 )$. Then $i_*(\gamma) \in I_{\wt{Z}}$.
\end{corol}

\begin{proof}

Let $\gamma$ be a class in $A^{\GL_3}_*  ( \ov{Z}_{(3,2)}
 )$. Since the restriction map
\[
 \wt{Z} \rrvert _{Z_{(3,2)}} \xrightarrow{\pi_1}
Z_{(3,2)}
\]
is a finite covering of order 3, applying Lemma \ref
{pushforward.symmetric} (setting $\Gamma$ to be the trivial group), we
get that $3 i_* (\gamma)$ is in the ideal of the alpha classes.
Moreover, by Proposition \ref{2gamma}, $2i_*(\gamma) \in I_{\wt
{Z}}$. Therefore $\gamma=3 i_*(\gamma)-2 i_*(\gamma) \in I_{\wt{Z}}$.
\end{proof}

\subsection{The Ideal $i_*  ( A^{\GL_3}_*  ( \ov
{Z}_{2} ) )$ Is Contained in $ ( \alpha_1, \alpha_2,
\alpha_3, \delta_{2}  )$}\label{Z2}

Let us consider the map
\[
\Pro(W_1) \times\Pro(W_2) \xrightarrow{
\pi_2} \ov{Z}_{2}.
\]
We denote by $h_1$ and $h_2$ two hyperplane classes corresponding to
the pull--back of hyperplane classes through the two projections:
\[
\xymatrix{
& \Pro(W_1) \times\Pro(W_2) \ar[dl]_{\pr_1} \ar[dr]^{\pr_2} & \\
\Pro(W_1) & & \Pro(W_2)
}
\]
By arguing as in Proposition \ref{prod.Wd.E}, we have a splitting
exact sequence of $A^*_{\GL_3}$-modules
\[
\xymatrix@R=10pt{
0 \ar[r] &  (P_{\{1\}}(x), P_{[2]}(y)  ) \\\ar[r] & A^*_{\GL
_3}[x,y] \ar[rr]_(.45){\hspace*{-6pt}\ev_{(h_1,h_2)}} & &  A^*_{\GL_3}
(\Pro(W_1) \times\Pro(W_2)  ) \ar[r] \ar@/_1pc/[ll]_{\psi}
& 0.
}
\]

First, let us determine the class $\delta_{2}$.

\begin{propos}\label{class.delta.2} We have the identity
\begin{equation}
\delta_2=21h_3^2 - 42h_3c_1
+ 9c_2 + 18c_1^2.
\end{equation}
Moreover, the classes $\alpha_1, \alpha_2, \alpha_3, \delta_2$ are
a set of independent generators for the ideal $ ( \alpha_1,
\alpha_2, \alpha_3, \delta_{2}  )$.
\end{propos}
\begin{proof} We refer to Section~\ref{universal-locus}. Since $\delta
_{2}=\delta_{ \{ 1,2\} }$, we evaluate formula (\ref
{delta.classes.formula}) for $\mu=\{ 1,2\}$ and $d=3$.

By preliminary computations we get
\begin{eqnarray*}
c^T_{\mathrm{top}} ( T_{Q_{(1,0,0)}} \Pro (W_1 ) )
&=& (l_2 - l_1) (l_3 - l_1),
\\
c^T_{\mathrm{top}} ( T_{Q_{(0,1,0)}} \Pro (W_1 ) )
&=& (l_1 - l_2) (l_3 - l_2),
\\
c^T_{\mathrm{top}} ( T_{Q_{(0,0,1)}} \Pro (W_1 ) )
&=& (l_1 - l_3) (l_2 - l_3),
\\
c^T_{\mathrm{top}} ( T_{Q_{(2,0,0)}} \Pro (W_2 ) )
&=& 4(l_2 - l_1)^2 (l_3 -
l_1)^2 (l_2 + l_3 - 2
l_1),
\\
c^T_{\mathrm{top}} ( T_{Q_{(0,2,0)}} \Pro (W_2 ) )
&=& 4(l_1 - l_2)^2 (l_3 -
l_2)^2 (l_1 + l_3 - 2
l_2),
\\
c^T_{\mathrm{top}} ( T_{Q_{(0,0,2)}} \Pro (W_2 ) )
&=& 4(l_1 - l_3)^2 (l_2 -
l_3)^2 (l_1 + l_2 - 2
l_3),
\\
c^T_{\mathrm{top}} ( T_{Q_{(1,1,0)}} \Pro (W_2 ) )
&=& -(l_1 - l_2)^2 (2l_3 -
l_1 - l_2) (l_3 - l_1)
(l_3 - l_2),
\\
c^T_{\mathrm{top}} ( T_{Q_{(1,0,1)}} \Pro (W_2 ) )
&=& -(l_1 - l_3)^2 (2l_2 -
l_1 - l_3) (l_2 - l_1)
(l_2 - l_3),
\\
c^T_{\mathrm{top}} ( T_{Q_{(0,1,1)}} \Pro (W_2 ) )
&=& -(l_2 - l_3)^2 (2l_1 -
l_2 - l_3) (l_1 - l_2)
(l_1 - l_3).
\end{eqnarray*}

Straightforward computations show the desired identity
\[
\delta_2= 21h_3^2 - 42h_3c_1
+ 9c_2 + 18c_1^2.
\]
Now, our goal is to prove that the classes $\alpha_1, \alpha_2,
\alpha_3, \delta_2$ are a set of independent generators for the ideal
$ ( \alpha_1, \alpha_2, \alpha_3, \delta_{2}  )$.

Notice that it suffices to consider the homogeneous ideal $(\alpha
_1,\alpha_2, \alpha_3, \delta_2)$ up to degree two. By working mod 2
we see that $\delta_2$ is not in the ideal $(\alpha_1,\alpha_2)$. On
the other hand, considering the classes mod 3, we see that $\alpha_2$
is not in the ideal $(\alpha_1,\delta_2)$.
\end{proof}

\begin{rmk}
From the identity
\[
2 \delta_2 = (5h_3 - 3 c_1)
\alpha_1 - 3 \alpha_2
\]
we have that $2 \delta_2 \in I_{\wt{Z}}$.
\end{rmk}

\begin{propos} We have the inclusion
\[
i_* ( A^{\GL_3}_* ( \ov{Z}_2 ) ) \subset (
\alpha_1, \alpha_2, \alpha_3,
\delta_{2} ).
\]
\end{propos}
\begin{proof}
Since we already know that $i_* ( A^G_*  ( \ov{Z}_3
) )$ is contained in $I_{\wt{Z}}$, we can restrict ourselves to
$Z_2$. First, notice that the restriction map
\[
U_2:= \Pro(W_1) \times\Pro(W_2)
\rrvert _{Z_2} \xrightarrow {\pi_2} Z_2
\]
is an isomorphism, and therefore $\pi_{2*}$ is an isomorphism of
$A^*_{\GL_3}$-modules. Consequently, it suffices to check that the
pushforwards of classes in $A^*_{\GL_3}(U_2)$ are contained in the
restriction of $ ( \alpha_1, \alpha_2, \alpha_3, \delta_{2}
 )$ to $Z_2$. Therefore, it suffices to show that
\[
\pi_{2*} ( A^*_{\GL_3} ( \Pro(W_1) \times
\Pro(W_2) ) ) \subset ( \alpha_1, \alpha_2,
\alpha_3, \delta_{2} ).
\]

We denote by $h_1$ and $h_2$ two hyperplane classes corresponding to
the pullback of hyperplane classes through two projections from $\Pro
(W_1) \times\Pro(W_2)$ to $\Pro(W_1)$ and to $\Pro(W_2)$. By
arguing as in Proposition \ref{prod.Wd.E}, we have a splitting exact
sequence of $A^*_{\GL_3}$-modules
\[
\xymatrix@R=10pt{
0 \ar[r] &  (P_{\{1\}}(x), P_{[2]}(y)  ) \\\ar[r] & A^*_{\GL
_3}[x,y] \ar[rr]_(.45){\hspace*{-6pt}\ev_{(h_1,h_2)}} & &  A^*_{\GL_3}
(\Pro(W_1) \times\Pro(W_2)  ) \ar[r] \ar@/_1pc/[ll]_{\psi}
& 0.
}
\]

The free $A^*_{\GL_3}$-module $\psi ( A^*_{\mathrm{GL}_3}  ( \Pro
(W_1) \times\Pro(W_2)  ) )$ is generated by the monomial
$x^{v_1}y^{v_2}$ such that $0 \leq v_1 \leq2$ and $0 \leq v_2 \leq5$.
Moreover, $\pi_2^*  ( h_3  ) = h_1 + h_2$, and by the
push--pull formula it suffices to evaluate $\pi_{2*}  ( h_1
 )$ and $\pi_{2*}  ( h_1^2  )$. We also have the identity
\begin{equation}
\label{h1.plus.h2} \pi_{2*} (h_1) + \pi_{2*}(h_2)
= h_3 \delta_2.
\end{equation}

Now, consider the following commutative diagram:
\[
\xymatrix{
\Pro(W_1) \times\Pro(W_2) \times\Pro(E) \ar[r]^-{\wt{\pi
}_2} \ar[d]^{\sigma_2} & \Pro(W_3) \times\Pro(E) \ar[d]^{\pi_1}\\
\Pro(W_1) \times\Pro(W_2) \ar[r]^-{\pi_2} & \Pro(W_3)
}
\]
where $\sigma_2$ is the natural projection, and $\wt{\pi}_2$ is the
lifting of $\pi_2$. We also have a splitting exact sequence of
$A^*_{\GL_3}$-modules
\[
\xymatrix@=10pt{
0 \ar[r] &  (P_{[1]}(x), P_{[2]}(y), P_{[1]}(-z)  ) \\\ar[r]
& A^*_{\GL_3}[x,y,z] \ar[rr]_(.37){\hspace*{6pt}\ev_{(h_1,h_2,t)}}  & &
A^*_{\GL_3}  (\Pro(W_1) \times\Pro(W_2) \times\Pro(E)
) \ar[r] \ar@/_1pc/[ll]_{\psi} & 0.
}
\]

Let $S_2 \subset\Pro(W_1) \times\Pro(W_2) \times\Pro(E)$ be the
locus of points of intersection between a linear form and a quadratic
form. Then $S_2$ is the complete intersection of the hypersurfaces
given by the equations
\begin{eqnarray*}
\sum_{v \in\N^3(1)} a_v X^v &=& 0,
\\
\sum_{v \in\N^3(2)} a_v X^v &=& 0.
\end{eqnarray*}
By using Lemma \ref{algebraic.lemma} we can restrict our computations
to $A^*_T$. We can therefore apply Lemma \ref{hyperplane.T.classes} to get
\[
[ S_2 ] = (h_1 + t) (h_2 + 2t).
\]
On the other hand, we have the inclusion $\wt{\pi}_2(S_2) \subset\wt
{Z}$. Let $\gamma$ be any multiple of the class $ [ S_2
]$. By commutativity of the diagram we have
\[
\pi_{2*} ( \sigma_{2*} (\gamma) ) \in I_{\wt{Z}}.
\]
Now, we choose $\gamma:= t \cdot [ S_2  ]$. A simple
computation shows that
\[
\psi(\gamma)=(2x + y - 2c_1)z^2 + (xy -
2c_2)z - 2c_3.
\]

Arguing as in Lemma \ref{lemma.point}, we get that $\sigma
_{2*}(\gamma)$ is the coefficient of $z^2$ evaluated at $(h_1,h_2)$,
that is, $\sigma_{2*}(\gamma) = 2h_1 + h_2 - 2c_1$. In particular, we get
\begin{equation}
\label{2h1.plus.h2} 2 \pi_{2*}(h_1) + \pi_{2*}(h_2)
- 2c_1\delta_2 \in I_{\wt{Z}}.
\end{equation}
Combining identities (\ref{h1.plus.h2}) and (\ref{2h1.plus.h2}), we
get
\[
\pi_{2*}(h_1) \in ( \alpha_1,
\alpha_2, \alpha_3, \delta_{2} ).
\]

To determine $\pi_{2*}(h_1^2)$, we first apply the push--pull formula
to get
\begin{equation}
\label{square.h1.plus.h2} \pi_{2*}(h_1^2) + 2
\pi_{2*}(h_1h_2) + \pi_{2*}(h_2^2)
= h_3^2\delta_2.
\end{equation}

Arguing as before, we have
\[
\sigma_{2*} ( h_2 t [S_2 ] ) =
2h_1h_2 + h^2_2 -
2c_1h_2,
\]
and, consequently,
\[
2\pi_{2*}(h_1h_2) + \pi_{2*}(h^2_2)
- 2c_1\pi_{2*}(h_2) \in I_{\wt{Z}}.
\]
Combining with identity (\ref{square.h1.plus.h2}), we get
\[
\pi_{2*}(h_1^2) \in ( \alpha_1,
\alpha_2, \alpha_3, \delta _{2} ).
\]
%
\end{proof}

\subsection{The Ideal $i_*  ( A^{\GL_3}_*  ( Z
) )$ Is Contained in $ ( \alpha_1, \alpha_2, \alpha_3,
\delta_{2}  )$}\label{conclusion}

Since we already know that $i_* ( A^G_*  ( \ov{Z}_2
) )$ is contained in $ ( \alpha_1, \alpha_2, \alpha_3,
\delta_{2}  )$, we can restrict to $Z_1$. Notice that $Z_1$ has
a stratification given by the locus of nodal cubics $Z'_1$ and the
locus of cubics with a cusp $Z''_1$. First of all, notice that the
restriction map
\[
 \wt{Z} \rrvert _{Z'_1} \xrightarrow{\pi_1}
Z'_1
\]
is an isomorphism, and therefore $\pi_{1*}$ is an isomorphism of
$A^*_{\GL_3}$-modules.

Let us now consider the restriction map:
\[
 \wt{Z} \rrvert _{Z''_1} \xrightarrow{\pi_1}
Z''_1.
\]
A simple calculation shows that the length of the fibers is two, and
since $\Char(k)>2$, it follows that the map is a Chow envelope.

We then argue as in Corollary \ref{tangent} to conclude the proof of
the following theorem.

\begin{theorem}\label{main.theorem}
Assume that the base field $k$ has the characteristic different from 2
and 3. Then
\[
i_* ( A^{\GL_3}_* (Z) ) = ( \alpha_1, \alpha_2,
\alpha_3, \delta_2 ),
\]
where
\begin{align*}
\begin{split} \alpha_1&=12(h_3 -
c_1),
\\
\alpha_2&= 6h_3^2 - 4h_3c_1
- 6 c_2,
\\
\alpha_3&= h_3^3 - h_3^2c_1
+ h_3c_2-9c_3,
\\
\delta_2 &=21h_3^2 - 42h_3c_1
+ 9c_2 + 18c_1^2. \end{split} %
\end{align*}
\end{theorem}

Main Theorem in the Introduction follows immediately.


\subsection*{Acknowledgements.}
We would like to thank the referees
for useful comments and suggestions.
\end{document}